\documentclass[a4paper,reqno,10pt]{amsart}
\usepackage{amsfonts} 
\usepackage{amssymb}
\usepackage{amsthm}
\usepackage{amsmath} 
\usepackage{mathrsfs} 
\usepackage{dsfont}
\usepackage{esint}

\usepackage[english]{babel}

\usepackage[left=3.5cm,right=3.5cm,top=3.5cm,bottom=3cm,headsep=0.7cm]{geometry}
\usepackage{pgf,tikz}
\usetikzlibrary{arrows} 
 
\usepackage[scriptsize,bf]{caption}
\usepackage{floatrow}

\usepackage{enumerate}
\usepackage{enumitem}

\usepackage{hyperref}
\usepackage{mathtools}
\mathtoolsset{showonlyrefs}

\theoremstyle{plain}
\begingroup
\theoremstyle{plain}
\newtheorem{theorem}{Theorem}[section]

\newtheorem{proposition}[theorem]{Proposition}

\theoremstyle{definition}
\newtheorem{definition}[theorem]{Definition}
\theoremstyle{remark}
\newtheorem{remark}[theorem]{Remark}

\endgroup

\theoremstyle{definition}
\theoremstyle{remark}

\numberwithin{equation}{section}

\newcommand{\myequation}{\begin{equation}}
\newcommand{\myendequation}{\end{equation}}
\let\[\myequation
\let\]\myendequation
  
\newlist{steps}{enumerate}{2}
\setlist[steps, 1]{wide=0pt, leftmargin=0cm, itemindent=15pt, listparindent=\parindent, label=\ul{\ul{\itshape Step \arabic*}}, ref={\itshape Step~\arabic*}}
\setlist[steps, 2]{wide=0pt, leftmargin=0cm, itemindent=20pt, listparindent=\parindent, label=\ul{\itshape Substep \arabic{stepsi}.\arabic*}, ref={\itshape Substep~\arabic{stepsi}.\arabic*}}

\newcommand{\R}{\mathbb{R}}
\newcommand{\N}{\mathbb{N}}

\mathchardef\emptyset="001F
\renewcommand{\d}{\mathrm{d}}
\newcommand{\de}{\partial}
\newcommand{\e}{\varepsilon}
\renewcommand{\tilde}{\widetilde}
\newcommand{\x}{{\times}}

\newcommand{\ul}{\underline}

\newcommand{\Tr}{\mathrm{Tr}}
\newcommand{\Lap}{\Delta}

\newcommand{\A}{\mathscr{A}} 
\renewcommand{\L}{\mathscr{L}}

\DeclareMathOperator{\Lip}{Lip}

\newcommand{\mynorm}{{\vert\kern-0.25ex\vert\kern-0.25ex\vert}}

\newcommand{\ie}{{\itshape i.e.}}
\newcommand{\eg}{{\itshape e.g.}}

\author[G. M. Coclite]{Giuseppe Maria Coclite}
\address[G. M. Coclite]{Politecnico di Bari, Dipartimento di Meccanica, Matematica e Management, Via E. Orabona 4, 70125 Bari,  Italy.}
\email[]{giuseppemaria.coclite@poliba.it}

\author[N. De Nitti]{Nicola De Nitti}
\address[N. De Nitti]{EPFL, Institut de Mathématiques, Station 8, 1015 Lausanne, Switzerland.}
\email{nicola.denitti@epfl.ch}

\author[F. Maddalena]{Francesco Maddalena}
\address[F. Maddalena]{Politecnico di Bari, Dipartimento di Meccanica, Matematica e Management, Via E. Orabona 4, 70125 Bari,  Italy.}
\email[]{francesco.maddalena@poliba.it}

\author[G. Orlando]{Gianluca Orlando}
\address[G. Orlando]{Politecnico di Bari, Dipartimento di Meccanica, Matematica e Management, Via E. Orabona 4, 70125 Bari,  Italy.}
\email[]{gianluca.orlando@poliba.it}

\author[E. Zuazua]{Enrique Zuazua}
\address[E. Zuazua]{Friedrich-Alexander-Universit\"at Erlangen-N\"urnberg, Department of Data Science, Chair for Dynamics, Control and Numerics (Alexander von Humboldt Professorship), Cauerstr. 11, 91058 Erlangen, Germany. 
	\newline \indent 	
	Chair of Computational Mathematics, Fundación Deusto,	Avenida de las Universidades, 24, 48007 Bilbao, Basque Country, Spain. 
	\newline \indent 
	Universidad Autónoma de Madrid, Departamento de Matemáticas, Ciudad Universitaria de Cantoblanco, 28049 Madrid, Spain.}
\email[]{enrique.zuazua@fau.de}

\title[]{Exponential convergence to steady-states for trajectories of a damped dynamical system modelling adhesive strings}
 
\keywords{Damped wave equation; adhesion phenomena; long-time asymptotics; decay.}
\subjclass[2020]{35L71, 35B40, 74H40, 74M15}

\begin{document}

\begin{abstract}
We study the global well-posedness and asymptotic behavior for a semilinear damped wave equation with Neumann boundary conditions, modelling a one-dimensional linearly elastic body interacting with a rigid substrate through an adhesive material.  The key feature of of the problem is that the interplay between the nonlinear force and the boundary conditions allows for a continuous set of equilibrium points. We prove an exponential rate of convergence for the solution towards a (uniquely determined) equilibrium point.
\end{abstract}

\maketitle

\setcounter{tocdepth}{1}
\tableofcontents 

\section{Introduction}
\label{sec:intro}


The qualitative and quantitative  study of the long-time behavior of solutions to dynamical systems has always been a challenging mathematical issue~\cite{Daf78}. The by now deep-rooted machinery for these problems can now deal with models arising from real phenomena. As usual, taking into account relevant physical features exposes to unavoidable additional difficulties. This is the case for the problem studied in this paper. 

Our problem falls into the class of semilinear evolution equations of the form
\[
  \de_{tt}^2 u(t,x) + \de_t u(t,x) - \Delta u(t,x) + f(u(t,x)) = 0 \, .
\]
In particular, we focus on the global well-posedness and the long-time behavior as $t \to +\infty$ of its solutions. More precisely, we intend to prove compactness of the trajectories in a suitable strong sense and deduce a precise estimate on the rate of the convergence to the limit trajectory. 

The peculiar structure of the nonlinearity $f(u)$ treated in the present paper is the novelty and the source of the distinguishing features and the related mathematical difficulties of our problem. These do not arise when polynomial nonlinearities are considered (see~\cite{Zua90, Zua91, Mar00}).  

In modelling the dynamics of an elastic string interacting with a rigid substrate through an adhesive layer, we are lead to analyze a nonlinearity mimicking possible attachment--detachment regimes. To take into account this kind of interactions, one considers a force $f(u)$ which vanishes when $u$ overcomes a critical threshold $u_*$ (see~\cite{MadPerPugTru09, CocFloLigMad17, CocFloLigMad19, CocDevFloLigMad23}). Specifically, we let $f(u) = \Phi'(u)$, where the potential $\Phi(u)$ is depicted in Figure~\ref{fig:Phik};  see~\eqref{def:Phik} for its precise definition.

\begin{figure}[H]
  \centering
  \includegraphics{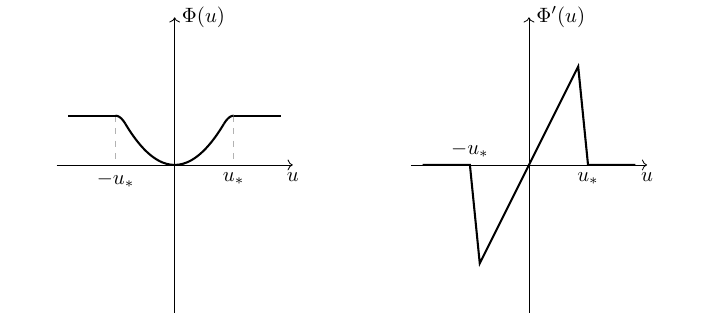}
	\caption{Plot of the potential $\Phi$ and the force $\Phi'$ as functions of $u$.}
	\label{fig:Phik}
\end{figure}

The model under investigation is ruled by the following initial boundary value problem:
\[  \label{intro:PDE}
\begin{cases}
\de^2_{tt} u(t,x) + \de_t u(t,x) -  \de^2_{xx} u(t,x)  + \Phi'(u(t,x)) = 0 \, , & (t,x) \in  (0,+\infty) \x (0,L)   \, , \\
 \de_x u(t,0) = \de_x u(t,L) = 0  \, , & t \in   (0,+\infty)  \, , \\
u(0,x) = u_0(x) \, , \  \de_t u(0,x) = v_0(x) \, , & x \in (0,L)\, .
\end{cases}
\]  
This describes the damped dynamics of a one-dimensional linearly elastic body, whose reference configuration is $(0,L)$, interacting with a rigid substrate through an adhesive material, acting through the force $\Phi'(u)$. If the displacement $u$ is small compared to $u_*$, this force is purely elastic; otherwise, if $|u| \geq u_*$, the adhesive material ceases to act on the elastic body. We stress that the adhesion process described here is reversible. In~\eqref{intro:PDE}, initial and Neumann boundary conditions are imposed.

The attachment--detachment process ruled by the nonlinear force $\Phi'(u)$ induces the natural question about the long-time behavior of such dynamics and, more specifically, whether a convergence to a stationary state occurs or switching between the two states persists. The presence of the damping term in~\eqref{intro:PDE} suggests that the former case should be the reasonable one. We answer this question quantitatively in the present paper as we explain below.

Our first task is to show global well-posedness of~\eqref{intro:PDE} in the energy space. This is achieved in Theorem~\ref{prop:well-posedness approximation initial data in energy space} by standard techniques. Note that the result is, in fact, obtained in arbitrary spatial dimension $d$. 

Our next task is to show convergence of trajectories as $t \to +\infty$. The structure of the dynamical system~\eqref{intro:PDE} suggests to move our steps in the framework of LaSalle's invariance principle (see~\cite{Las67, Hal69} and also~\cite[Theorem 9.2.3]{CazHar}). In this perspective, one expects that $u(t) \to u_\infty$ as $t \to +\infty$, where $u_\infty$ is a solution to the stationary equation 
\[ \label{eqintro:stationary}
  \begin{cases}
     -  \de^2_{xx} u_\infty (x) + \Phi'(u_\infty(x)) = 0 \, , & x \in  (0,L) \, , \\
     \de_x u_\infty(0) = \de_x u_\infty(L) = 0  \, .
    \end{cases}
\]
In fact, the first step in this direction is reached in Proposition~\ref{prop:omega-limit set} (proven in any dimension $d \geq 1$), where we show that the accumulation points of the trajectories $\{u(t)\}_{t \geq 0}$ satisfy~\eqref{eqintro:stationary} via a compactness argument.
We bring the reader's attention to the decisive role played here by the interplay between the set of critical points of $\Phi$ and the kernel of the Laplace operator with Neumann boundary conditions, which makes the arguments intriguing. Indeed, if homogeneous Dirichlet conditions were imposed instead, the stationary problem 
\[ 
  \begin{cases}
     -  \de^2_{xx} u_\infty (x) + \Phi'(u_\infty(x)) = 0 \, , & x \in  (0,L) \, , \\
     u_\infty(0) = u_\infty(L) = 0  \, ,
    \end{cases} 
\]
would admit a unique trivial solution---the constant value 0. The same would occur in the case of Neumann boundary conditions with a force only vanishing at $0$, \eg, 
\[ 
  \begin{cases}
     -  \de^2_{xx} u_\infty (x) + u_\infty |u_\infty|^{p-1} = 0 \, , & x \in  (0,L) \, , \\
     \de_x u_\infty(0) = \de_x u_\infty(L) = 0  \, ,
    \end{cases} \, , \quad p > 1 \, .
\]
In those cases, uniqueness of the accumulation point, combined with the compactness argument, would force the desired convergence; see~\cite{Zua90, Zua91, Mar00} for results in that direction. The same conclusion holds true also in the case of multiple, yet discrete, rest points, as pointed out in~\cite{Dic76, Bal78, Web79, Bal04}. Here, in contrast, the set of solutions to problem~\eqref{eqintro:stationary} is given by constant functions valued in $\{\Phi' = 0\} = (-\infty, -u_*] \cup \{0\} \cup [u_*, +\infty)$, allowing for a continuous set of possible choices for the limit profiles. Therefore, the compactness argument alone does not suffice to infer convergence of the whole trajectory as $t \to +\infty$ to a uniquely determined limit profile. Even when proving this convergence in the case of a continuous set of equilibria is possible, finer techniques are required, see, \eg, \cite{Zua88} for this phenomenon arising in a dissipative system when the Laplacian is replaced by a differential operator with a one-dimensional kernel. Here, in Theorem~\ref{prop:long-time limit}, we accomplish to prove that the initial conditions~\eqref{intro:PDE} enforce the selection of a unique limit profile $u_\infty \in \{\Phi' = 0\}$, thus allowing us to conclude the convergence $u(t) \to u_\infty$ as $t \to +\infty$. The proof of Theorem~\ref{prop:long-time limit} is carried out in any dimension $d \geq 1$ and its main ingredient is the existence of an auxiliary Lyapunov functional, introduced in~\eqref{def:J}, inspired by~\cite{Har86}.

Our last, yet fundamental, task is to quantify the rate of convergence for such asymptotic dynamics. The interplay of nonlinearities and differential operators with a nontrivial kernel might affect the rate of decay of dissipative dynamical systems, as shown in some cases with the existence of slow solutions decaying polynomially fast in~\cite{Har05,GhiGobHar16-2,GhiGobHar16-1}.  In Theorem~\ref{prop:decay-rate}, we prove that this phenomenon is ruled out in the model studied in this paper, by showing that the convergence $u(t) \to u_\infty$ as $t \to +\infty$ occurs in an exponential fashion, \ie, $\|u(t,\cdot) - u_\infty\|_{H^1((0,L))} \leq M e^{- \kappa t}$. The arguments used in the proof rely on the compact embedding $H^1((0,L)) \subset \subset C([0,L])$, true only in dimension $d = 1$. This allows us to single out only one of the two possible attachment--detachment regimes for time large enough, thus allowing us to study separately the decay rate of solutions to the equations 
  \begin{alignat*}{2}
    & \partial_{tt}^2 u(t,x) + \de_t u(t,x) - \de_{xx}^2 u(t,x)  = 0, \qquad && (t,x) \in (0,+\infty) \x (0,L)  \, , \\
    & \partial_{tt}^2 u(t,x) + \de_t u(t,x) - \de_{xx}^2 u  + 2 u(t,x) = 0, \qquad && (t,x) \in (0,+\infty)\x (0,L) \, .
  \end{alignat*}
Trajectories of both dynamics share the exponential decay rate to the equilibrium points. This result is proved through Gr\"onwall type inequalities on suitable perturbations of the energy functionals, in the spirit of~\cite{HarZua88}, giving a positive answer to the question raised in the paper.  

We conclude by commenting on an interesting question concerning the decay rate and the shape of the force $\Phi'(u)$. If the rate of the exponential decay $M e^{-\kappa t}$ is uniform with respect to the slope of the decreasing part of the force $\Phi'(u)$, one can think of treating the degenerate case of discontinuous force, \eg, $\Phi'(u) = 2u$ for $|u| < u_*$ and $\Phi'(u) = 0$ for $|u| \geq u_*$. This models an abrupt  attachment--detachment phenomenon. Answering to this question is not in the scope of the paper. However, in this perspective, we study an ODE related to this phenomenon in Appendix~\ref{app:ODE}, showing that the  uniform exponential decay occurs.

All the precise statements of the results contained in the paper are collected in Section~\ref{sec:main}.

\section{Main results}
\label{sec:main}

In this section we present the model and the main results proven in the paper. 

\subsection{The model} Let $d \geq 1$ be the dimension. We will explicitly declare for which results $d = 1$ is needed. We fix $\Omega \subset \R^d$, a bounded, open, and connected set, additionally with $C^2$ boundary when $d \geq 2$.\footnote{This regularity requirement is motivated by the need to define traces on $\de \Omega$, exploit the compact embedding $H^1(\Omega) \subset L^2(\Omega)$, and apply elliptic regularity theory up to the boundary.} We let $\nu$ denote the exterior normal derivative to $\de \Omega$. 

We give a precise definition of the nonliner term in the equation. Let us fix a threshold $u^* \in (0,+\infty)$ and a potential $\Phi \in C^1(\R;\R)$ defined as follows:
\[ \label{def:Phik}
\Phi(u)  := \begin{dcases}
u^2 \, , & \text{if } |u| \leq u_* - \frac{1}{\sigma} \, , \\
u_*\Big(u_* - \frac{1}{\sigma}\Big) - (u_* \sigma-1) (u_*-u)^2  \, , & \text{if } u_* - \frac{1}{\sigma} \leq u \leq u_* \, , \\
u_*\Big(u_* - \frac{1}{\sigma}\Big) - (u_* \sigma-1) (u_*+u)^2  \, , & \text{if } - u_* \leq u \leq -u_* + \frac{1}{\sigma} \, , \\
u_*\Big(u_* - \frac{1}{\sigma}\Big) \, , & \text{if }  u_* \leq |u| \, ; 
\end{dcases}
\]
see Figure~\ref{fig:Phik}. The specific form of the potential $\Phi$ is not significant; the following properties are the relevant ones: 
\begin{enumerate}[label=(P\arabic*)]
	\item \label{it:p1} $|\Phi| \leq u_*^2$;
	\item \label{it:p2} $|\Phi'| \leq 2$ and $\Phi'$ is globally Lipschitz with $\Lip(\Phi') \leq 2(u_*\sigma-1)$;
	\item \label{it:p3} $\Phi(u) - u^2$ is concave;
	\item \label{it:p5} $u \Phi'(u) \geq 0$.
\end{enumerate}

We study the following initial boundary value problem
\[  \label{eq:PDE}
\begin{cases}
\de^2_{tt} u - \Lap u + \de_t u + \Phi'(u) = 0 \, , & (t,x) \in  (0,+\infty) \x \Omega   \, , \\
\nabla u \cdot \nu = 0  \, , & (t,x) \in   (0,+\infty) \x \de \Omega \, , \\
u(0,x) = u_0(x) \, , \  \de_t u(0,x) = v_0(x) \, , & x \in  \Omega   \, .
\end{cases}
\]

\subsection{Global well-posedness}

To establish the well-posedness result, it is convenient to interpret \eqref{eq:PDE}  as the system of ODEs in $H^1(\Omega) \x L^2(\Omega)$ 
\[   \label{eq:recast PDE full}
\begin{dcases}
\frac{\d u}{\d t}(t) - v(t) = 0 \, , &  t \in (0,+\infty) \, ,  \\
\frac{\d v}{\d t}(t) - \Lap u(t) + v(t) + u(t) - u(t)  + \Phi'(u(t)) =  0 \, , & t \in (0,+\infty)  \, ,   \\
u(0) = u_0 \, , \  v(0) = v_0 \, .
\end{dcases}
\] 

In Section~\ref{sec:wp-approx} we prove the following result.

\begin{theorem}[Well-posedness with initial  data in the energy space]
	\label{prop:well-posedness approximation initial data in energy space}
	Let $d \geq 1$. Let $U_0 = (u_0,v_0)$ and assume that 
	\[ \label{eq:initial data in energy space}
	u_0 \in H^1(\Omega) \quad \text{ and } \quad v_0 \in L^2(\Omega) \, .
	\]
	Then there exists a unique solution $(u,v) \in C\big([0,+\infty); H^1(\Omega) \times L^2(\Omega)\big)$ to~\eqref{eq:recast PDE full} with initial datum $u(0) = u_0$ and $v(0) = v_0$ (see Proposition~\ref{prop:well posedness nonlinear abstract} below for the precise definition of solution). 
\end{theorem}
  
\begin{remark} \label{rem:weak solution}
	A solution $u$ in the sense of Theorem~\ref{prop:well-posedness approximation initial data in energy space} is also a weak solution to~\eqref{eq:recast PDE full} in the following sense: $u \in C([0,+\infty);H^1(\Omega)) \cap C^1([0,+\infty);L^2(\Omega))$, and it satisfies the ODE 
	\[ \label{eq:ODE in L2}
	\begin{dcases}
	\frac{\d u}{\d t}(t) - v(t) = 0 \, , & \text{for } t \in (0,+\infty) \, ,  \\
	u(0) = u_0 \, ,
	\end{dcases}
	\]
	and 
	\[ \label{eq:weak vk}
	\begin{split}
	0 & = \langle v_0, \varphi(0) \rangle_{L^2(\Omega)} +  \int_0^{+\infty}\Big\langle v(t), - \frac{\d \varphi}{\d t}(t) + \varphi(t)\Big\rangle_{L^2(\Omega)}  \, \d t + \int_0^{+\infty} \langle \nabla u(t), \nabla \varphi(t) \rangle_{L^2(\Omega)} \, \d t \\
	& \quad + \int_0^{+\infty} \langle \Phi'(u(t)) , \varphi(t) \rangle_{L^2(\Omega)} \, \d t \, ,
	\end{split}
	\]
	for every $\varphi \in C^\infty_c(\R \x \R^d)$. We show this additional fact in the proof of Theorem~\ref{prop:well-posedness approximation initial data in energy space}.
\end{remark}

\subsection{Energy balances} We stress that a key role in our analysis is played by the \emph{energy functional} $E \colon H^1(\Omega) \x L^2(\Omega) \to [0,+\infty)$ defined~by 
\[ \label{def:E}
E(u,v) := \frac{1}{2}\|v\|^2_{L^2(\Omega)} + \frac{1}{2} \|\nabla u\|^2_{L^2(\Omega)} + \|\Phi(u)\|_{L^1(\Omega)}  \, , \quad \text{for all } (u,v) \in H^1(\Omega) \x L^2(\Omega) \, 
\]
and the auxiliary functional $J \colon H^1(\Omega) \x L^2(\Omega) \to [0,+\infty)$ defined~by
\[ \label{def:J}
J(u,v) := \frac{1}{2}\|u\|^2_{L^2(\Omega)} + \langle u, v \rangle_{L^2(\Omega)} \, , \quad \text{for all } (u,v) \in H^1(\Omega) \x L^2(\Omega)   \, .
\]

In Section~\ref{sec:wp-approx} we prove the following energy balances.
	
\begin{theorem}[Energy balances] \label{th:energy}	
  Let $d \geq 1$. Let $U_0 = (u_0,v_0)$ and assume that \eqref{eq:initial data in energy space} is satisfied. 
Then  the unique solution to~\eqref{eq:PDE} obtained in Theorem~\ref{prop:well-posedness approximation initial data in energy space} satisfies the following equalities: 
	\[ 
	\label{eq:approximate energy estimate II}
	E(u(t),\de_t u(t)) + \int_0^t \|\de_t u(s)\|_{L^2(\Omega)}^2 \, \d s = E(u_0,v_0) \, , \quad \text{for every } t \in [0,+\infty)\, , 
	\]
	and
	\[ \label{eq:estimate on Ik II}
	\begin{split}
	& J(u(t),\de_t u(t)) + \int_0^t \Big( \|\nabla u(s)\|^2_{L^2(\Omega)} + \langle u(s), \Phi'(u(s)) \rangle_{L^2(\Omega)} \Big) \, \d s \\
	& \quad = J(u_0,v_0) + \int_0^t \|\de_t u(s)\|_{L^2(\Omega)}^2 \, \d s \, , \quad \text{for every } t \in [0,+\infty) \, . 
	\end{split}
	\]
	
\end{theorem}

\begin{remark}
  Theorem~\ref{th:energy} shows, in particular, that $E$ is  nonnegative and decreasing on trajectories. The same holds true for 
  \[ \label{eq:J + int}
    J(u(t),\de_t u(t)) + \int_t^{+\infty} \|\de_t u(s)\|_{L^2(\Omega)}^2 \, \d s \, .
  \] 
  Both of these facts will play a crucial role in the argument of this paper inspired by LaSalle's invariance principle. 
\end{remark}

\subsection{Long-time asymptotics}
We then turn to studying the asymptotic behavior of the solution. Our first result in this direction is qualitative and is proven in Section~\ref{sec:decay-approx}. 

\begin{theorem}[Long-time asymptotics] \label{prop:long-time limit}
	Let $d \geq 1$. Assume that $u_0, v_0$ satisfy~\eqref{eq:initial data in energy space}. Let $(u,v)$ be the unique solution to~\eqref{eq:PDE} obtained in Theorem~\ref{prop:well-posedness approximation initial data in energy space}. Then there exists $u_\infty$ constant a.e.\ in $\Omega$ with $u_\infty \in \{\Phi' = 0\}$ such that $u(t) \to u_\infty$ in $H^1(\Omega)$ and $\de_t u(t) \to 0$ in $L^2(\Omega)$ for $t \to +\infty$. 
\end{theorem}

We remark that constant functions valued in $\{\Phi' = 0\}$ are the weak solutions to the stationary problem
  \[  
    \begin{cases}
         - \Lap u  + \Phi'(u ) = 0 \, , & x \in  \Omega   \, , \\
        \nabla u  \cdot \nu = 0  \, , & x \in    \de \Omega \,  .
    \end{cases}
\]

\begin{remark}
  Even though not explicitly quantified, the long-time limit $u_\infty$ in Theorem~\ref{prop:long-time limit} depends only the unique long-time limit of~\eqref{eq:J + int}, which in turn only depends on the initial data. 
\end{remark}

\subsection{Exponential decay rate}
Our final result is the exponential decay of the solutions proven in Section~\ref{sec:rate-approx}. For this, we need to exploit the compact  embedding of $H^1(\Omega) \hookrightarrow C(\Omega)$, which holds in one space-dimension, and exclude the critical case $u_\infty \equiv \pm u_*$. 

\begin{theorem}[Exponential decay rate]\label{prop:decay-rate}
	Let us assume that $d =1$ and $\Omega = (0,L) \subset \R$. Given initial data $u_0, v_0$ satisfying~\eqref{eq:initial data in energy space}, let $(u,v)$ be the unique solution to~\eqref{eq:PDE} obtained in Theorem~\ref{prop:well-posedness approximation initial data in energy space}. Let  $u_\infty$ be as in Theorem~\ref{prop:long-time limit}. Let us assume that $u_\infty = 0$ or $|u_\infty| > u_*$. Then 
	\[ \label{eq:exponential decay approximate solutions}
	\|u(t) - u_{\infty}\|_{H^1(\Omega)} + \|\de_t u(t)\|_{L^2(\Omega)} \leq M_\Phi e^{-\kappa t} \, ,  
	\]
	for some $M_\Phi > 0$ (possibly depending on the Lipschitz constant of $\Phi'$) and $\kappa > 0$. 
\end{theorem}

\begin{remark}[Dependence of the decay rate on the Lipschitz constant of $\Phi'$]\label{rk:sigma}
	In the proof of Theorem \ref{prop:decay-rate}, the decay rates are given for time $t > T_\sigma$, depending on the constant $\sigma$ in the definition \eqref{def:Phik}. However, it is not clear whether $T_\sigma$ is bounded independently of $\sigma$ or $T_\sigma \to + \infty$ as $\sigma \to 0^+$. This issue is addressed in Appendix \ref{app:ODE} for a finite-dimensional system.  
\end{remark}

\section{Preliminary results on abstract evolution problems and semigroups}
\label{sec:prelim}

The well-posedness theory for semilinear evolution problems of the form 
\[\label{eq:abstract nonlinear}
\begin{dcases}
\frac{\d U}{\d t}(t)  + \A U(t) + F(U(t))= 0 \, , & t \in  (0,+\infty)  \, , \\
U(0) = U_0 \, ,  &   
\end{dcases}
\] 
where $\A$ is a maximal monotone operator, is based on a well-known machinery (see, \eg, \cite{Har,CazHar,Brez}). We collect some details in this section since some tools will be used later. 

First, we recall the following definition.

\begin{definition}[Maximal monotone operators] \label{def:maximal monotone}
	Let $H$ be a Hilbert space, let $D(\A) \subset H$ be a vector subspace, and let $\A \colon D(\A) \to H$ be a linear operator. The linear operator $\A$ is said to be \emph{maximal monotone} if the following two conditions are satisfied:
	\begin{enumerate}
		\item[1.] $\langle \A U, U \rangle_H \geq 0$ for all $U \in D(\A)$;
		\item[2.] there exists $\lambda > 0$ such that $( \mathrm{Id} + \lambda \A ) D(\A) = H$.   
	\end{enumerate} 
\end{definition}

The starting point for the analysis of \eqref{eq:abstract nonlinear} is the well-posedness for the linear homogeneous problem. 

\begin{proposition}[Semigroup of contractions and abstract homogeneous Cauchy problem]  \label{prop:semigroup of contractions}
	Let $\A \colon D(\A) \to H$ be a maximal monotone operator. Then $\A$ generates a continuous semigroup of contractions $\{S_\A(t)\}_{t  \geq 0}$, \ie, for every $t \geq 0$, $S_\A(t) \colon H \to H$ is a linear operator with $\|S_\A(t)\|_H \leq 1$, $S_\A(s+t) = S_\A(s)S_\A(t)$, $S_\A(0) = \mathrm{Id}$, and $\lim_{t \to 0^+} \| S_\A(t)U_0 - U_0\|_{H} = 0$ for every $U_0 \in H$. Moreover, if $U_0 \in D(\A)$, then $U(t) := S(t)U_0 \in C^1([0,+\infty);H) \cap C([0,+\infty);D(\A))$ is the unique solution to 
	\[ \label{eq:homogeneous ODE}
	\begin{dcases}
	\frac{\d U}{\d t}(t)  + \A U(t) = 0 \, , & t \in  (0,+\infty)  \, , \\
	U(0) = U_0 \, .  &   
	\end{dcases}
	\] 
\end{proposition}
\begin{proof}
	The result is classical: see, \eg, \cite[Theorem~I.2.2.1 \& Remark~I.2.2.3]{Har}. 
\end{proof}

Owing to the previous result, the curve $U(t) = S_\A(t)U_0$ is interpreted as the solution to~\eqref{eq:homogeneous ODE} even when the initial datum $U_0 \in H$ (rather than the smaller space $D(\A)$). 

To deal with the (linear) inhomogeneous problem, we need to resort to \emph{Duhamel's formula}.

\begin{proposition}[Duhamel's formula and abstract inhomogeneous Cauchy problem] \label{prop:solution less regular data}
	Let $\A \colon D(\A) \to H$ be a maximal monotone operator and let $\{S_\A(t)\}_{t \geq 0}$ be the continuous group of contractions given by Proposition~\ref{prop:semigroup of contractions}.  Let $T > 0$ and let $f \in W^{1,1}([0,T]; H)$. Let $U_0 \in D(\A)$. Then the curve defined by  
	\[ 
	U(t) := S_\A(t)U_0 - \int _0^t S_\A(t-s) f(s) \, \d s \, , \quad t \in [0,T]
	\]  
	satisfies $U \in C^1([0,T];H) \cap C([0,T];D(\A))$ and is the unique solution to 
	\[
	\begin{dcases} \label{eq:inhomogeneous ODE}
	\frac{\d U}{\d t}(t)  + \A U(t) + f(t)= 0 \, , & t \in  (0,T)  \, , \\
	U(0) = U_0 \, .  &   
	\end{dcases}
	\] 
\end{proposition}
\begin{proof}
	The result is obtained by approximating the initial datum with initial data in $D(\A)$ and applying Proposition~\ref{prop:semigroup of contractions}. For details, see, \eg, \cite[Proposition~4.1.6]{CazHar}.
\end{proof}
As above, thanks to Duhamel's formula one interprets the curve $U(t) := S_\A(t)U_0 - \int _0^t S_\A(t-s) f(s) \, \d s$ as the solution to~\eqref{eq:inhomogeneous ODE} even when $U_0 \in H$ (rather than the smaller space $D(\A)$).

We conclude the preliminary results by recalling the well-posedness for semilinear problems with a Lipschitz-continuous nonlinearity. Duhamel's formula allows us to give a notion of solution here as well. 

\begin{proposition}[Well-posedness of an abstract semilinear problem] \label{prop:well posedness nonlinear abstract}
	Let $\A \colon D(\A) \to H$ be a maximal monotone operator and let $\{S_\A(t)\}_{t \geq 0}$ be the continuous group of contractions given by Proposition~\ref{prop:semigroup of contractions}.  Let $F \colon H \to H$ be Lipschitz continuous and let $U_0 \in H$. Then there exists a unique curve $U \in C([0,+\infty);H)$ such that 
	\[ \label{eq:ODE in integral form}
	U(t) = S_\A(t) U_0 - \int _0^t S_\A(t-s) F(U(s)) \, \d s \, , \quad t \in [0,+\infty)\, .
	\]
	If, in addition, $U_0 \in D(\A)$, then $U \in C^1([0,+\infty);H) \cap C([0,+\infty);D(\A))$ and it satisfies
	\[ 
	\begin{dcases}
	\frac{\d U}{\d t}(t)  + \A U(t) +  F(U(t)) = 0 \, , & t \in  (0,+\infty)  \, , \\
	U(0) = U_0 \, .  &   
	\end{dcases}
	\]
\end{proposition} 
\begin{proof}
	The proof of the existence and uniqueness claim is based on Banach's fixed-point theorem (see, \eg, \cite[Proposition 4.3.3]{CazHar}). For completeness, we provide a sketch of the proof below. 
	
Following \cite[Theorem 7.3]{Brez}, we note that $C([0,+\infty) ; H)\eqqcolon E$ is a Banach space when endowed with the exponentially-weighted norm $\|U\|_{E}:=\sup _{t \geqslant 0} e^{-\gamma t}\|U(t)\|_{H}$.

We fix $\gamma \ge  \mathrm{Lip}(F)$ and apply Banach's fixed point theorem to the map 
$$T: V \mapsto  S_\A(t) U_0 - \int _0^t S_\A(t-s) F(V(s))$$
  in the space $X\coloneqq \{V \in C([0,+\infty); H): \, \|V\|_E \le \|U_0\|_H\}$. 
 First, we check that $T$ is a self mapping of $X$: 
 \begin{align*}
\|V(t)\|_H &\le \|U_0\|_H + \int_0^t \mathrm{Lip}(F) \|V(s)\|_H \, \d s \\
&\le \|U_0\|_H \exp(\mathrm{Lip}(F)t) \, ,
 \end{align*}
which yields 
\begin{align*}
\|V\|_{E} \le \|U_0\|_H \underbrace{\exp((\mathrm{Lip}(F) - \gamma)t)}_{\le 1} \,. 
\end{align*}
	
Second, we check that $T$ is a contraction: 
\begin{align*}
\|T(U)(t)-T(V)(t)\|_{H} &\leq \int_{0}^{t} \underbrace{\|S_{\mathscr{A}}(t-s)\|}_{\le c < 1}\mathrm{Lip}(F)\|U(t)-V(t)\|_H \\
&  \le   c \, \mathrm{Lip}(F) \,  \int_{0}^{t}  e^{\gamma s}\, \d s \,  \|U-V\|_E \\
&  \le  c \frac{\mathrm{Lip}(F)}{\gamma} (e^{\gamma t} - 1)\, \|U - V\|_E,
\end{align*}	
	which yields 
\begin{align*}
\|T(U)(t)-T(V)(t)\|_{E} &\leq    \underbrace{c \frac{\mathrm{Lip}(F)}{\gamma} (1- e^{\gamma t})}_{\le c < 1} \, \|U-V\|_E.	
	\end{align*} 
By the arbitrariness of the time-horizon $T>0$, we may thus deduce the claimed global well-posedness result. The additional regularity for $U_0 \in D(\A)$ is proven in \cite[Proposition 4.3.9]{CazHar}.
\end{proof}

\begin{remark}[Dynamical system] \label{rmk:dynamical system operator}
	Let us consider the family of (nonlinear) operators $\{T(t)\}_{t \geq 0}$ where $T(t) \colon H \to H$ defined as follows: For every $U_0 \in H$, let $U(t)$ be the unique solution to~\eqref{eq:ODE in integral form} provided by Proposition~\ref{prop:well posedness nonlinear abstract} and let $T(t)(U_0) := U(t)$. It satisfies the following properties: 
	\begin{enumerate}[label=(D\arabic*)]
		\item \label{it:1} $T(0) = \mathrm{Id}$;
		\item \label{it:2} $T(t)$ is continuous for every $t \geq 0$;
		\item \label{it:3} $T(t + s) = T(t) \circ T(s)$ for every $t,s \geq 0$;
		\item \label{it:4} $t \mapsto T(t)(U_0)$ is continous for every $U_0 \in H$.
	\end{enumerate}
	Following the notation of \cite[Definition 9.1.1]{CazHar}, we say that $\{T(t)\}_{t \ge 0}$ is a \emph{dynamical system on $H$}.
	Properties~\ref{it:1} and~\ref{it:4} are consequences of the definition and the continuity in time of solutions to the equation. Property~\ref{it:3} follows from the uniqueness of solutions.  Property~\ref{it:2} amounts to continuous dependence on the initial data for the solution of the PDE, which follows from the Gr\"onwall-type argument in \cite[Proposition 4.3.7]{CazHar}.
\end{remark}

\section{Proof of the well-posedness result}
\label{sec:wp-approx}

\subsection{ODE in Hilbert space} 
\label{ssec:ODE formulation} 

To apply the general well-posedness theory for semilinear evolution problems illustrated in Section~\ref{sec:prelim}, it is convenient to recast the Cauchy problem~\eqref{eq:recast PDE full} in the form~\eqref{eq:abstract nonlinear}. We let $H = H^1(\Omega) \x L^2(\Omega)$ be endowed with the scalar product 
\[ \label{eq:scalar product in H}
\begin{split}
\langle (u,v), (\tilde u, \tilde v) \rangle_H & := \int_{\Omega} u \tilde u \, \d x + \int_{\Omega} \nabla u \cdot \nabla \tilde u \, \d x  + \int_{\Omega} v \tilde  v \, \d x \\
& = \langle u,\tilde u\rangle_{L^2(\Omega)} + \langle \nabla u,\nabla  \tilde u\rangle_{L^2(\Omega)} + \langle v,\tilde v\rangle _{L^2(\Omega)} \quad \text{for } (u, v), (\tilde u, \tilde v) \in H \, ,
\end{split}
\]
and we consider the linear operator $\A \colon D(\A) \to H$ 
\[ \label{eq:L}
\left\{
\begin{aligned}
& D(\A) := \{ (u,v) \in H \ : \ u \in H^2(\Omega) \text{ with } \Tr(\nabla u \cdot \nu) = 0 \text{ on } \de \Omega  \, ,  \  v \in H^1(\Omega) \} \, , \\
& \A(u,v) := (-v, - \Lap u + v + u) \, , \text{ for } (u,v) \in D(\A) \, .
\end{aligned}
\right.    
\]

In this framework, problem~\eqref{eq:recast PDE full} reads
\[    \label{eq:recast PDE compact}
\begin{dcases}
\frac{\d U}{\d t}(t)  + \A U(t) +  F(U(t)) = 0 \, , & t \in  (0,+\infty)  \, , \\
U(0) = U_0 \, ,  &   
\end{dcases}
\] 
with $U = (u,v)$ and $U_0 = (u_0, v_0)$.

\begin{remark}[$\mathscr A$ is maximal monotone] \label{rmk:maximal monotone}
	The operator $\A$ is maximal monotone according to Definition~\ref{def:maximal monotone}. To see this, integrating by parts we obtain for all $(u,v) \in D(\A)$
	\[ \label{eq:skew-adjoint}
	\begin{split}
	\langle \A(u,v), (u,v) \rangle_H &  = \langle (-v, - \Lap u + v + u), (u,v) \rangle_H   \\
	& = - \int_\Omega v u \, \d x - \int_\Omega \nabla v \cdot \nabla u \, \d x - \int_\Omega \Lap u v \, \d x  + \int_\Omega v^2 \, \d x + \int_\Omega u v \, \d x \\
	& = \int_\Omega v^2 \, \d x \geq 0 \, .
	\end{split}
	\]
	Moreover, for every $(f,g) \in H$ there exists $(u,v) \in D(\A)$ such that $(u,v) + \A(u,v) = (f,g)$. Indeed, this condition is equivalent to the existence of solutions to the problem 
	\[
	\begin{cases}
	u - v = f \, , \\ 
	v - \Lap u + v + u = g \, ,
	\end{cases}
	\iff 
	\begin{cases}
	v = u - f \, , \\ 
	3 u  - \Lap u  = 2f + g \, .
	\end{cases}
	\]
	Since $2f + g \in L^2(\Omega)$, there exists a unique solution to 
	\[
	\begin{cases}
	3 u  - \Lap u  = 2f + g \, , & x \in \Omega \, , \\
	\nabla u \cdot \nu = 0 \, , & x \in \de \Omega \, ,
	\end{cases}
	\]
	which, by elliptic regularity theory, belongs to $H^2(\Omega)$ and satisfies the previous equation in a classical sense. Then, setting $v = u - f$, we have $v \in H^1(\Omega)$, hence $(u,v) \in D(\A)$. This proves that $\A$ is maximal monotone.
\end{remark}

\begin{remark}[Lipschitz continuity of the nonlinearity]
	We observe that the nonlinear perturbation $F \colon H \to H$ given by
	\[ \label{eq:Fk}
	F(U) = (0, -u + \Phi'(u)) \, , \quad \text{for } U = (u,v) \in H  \, , 
	\]
	is Lipschitz continuous. Indeed, the continuity of $\Phi'$ yields 
	\[ \label{eq:F lipschitz}
	\begin{split}
	\|F(U) - F(\tilde U)\|_H & = \| -u + \Phi'(u) + \tilde u - \Phi'(\tilde u) \|_{L^2(\Omega)} \\
	& \leq \|u - \tilde u\|_{L^2(\Omega)}  + \Lip(\Phi') \|u - \tilde u\|_{L^2(\Omega)} \\
	& \leq (1 + \Lip(\Phi')) \|U -  \tilde U \|_{H} \, ,
	\end{split}
	\]
	for every $U = (u, v), \ \tilde U = (\tilde u, \tilde v) \in  H$. Notice that the Lipschitz constant explodes in the limit $k \to +\infty$. 
\end{remark}

\subsection{Solutions with regular initial data}
\label{ssec:sol-regd}

In this subsection, we study the well-posedness of~\eqref{eq:recast PDE compact} under higher regularity assumptions on initial data. Specifically, we assume  that 
\[ \label{eq:regularity initial data}
  u_0 \in H^2(\Omega) \text{ with } \Tr(\nabla u_0 \cdot \nu) = 0 \text{ on } \de \Omega\, , \quad   \quad v_0 \in H^1(\Omega) \, ,
\]
\ie, $(u_0,v_0) \in D(\A)$.

Thanks to the observations in Subsection~\ref{ssec:ODE formulation}, we obtain the following result.

\begin{proposition}[Well-posedness of the problem with regular initial data] \label{prop:well-posedness approximation with regular data}
  Let $U_0 = (u_0,v_0)$ and assume that $u_0,v_0$ satisfy~\eqref{eq:regularity initial data}. Then there exists a unique solution $U = (u,v) \in C^1([0,+\infty);H) \cap C([0,+\infty); D(\A))$ to~\eqref{eq:recast PDE compact} with initial datum $u(0) = U_0$, in the sense of Proposition~\ref{prop:solution less regular data}. In particular, $U = (u, v)$ satisfies~\eqref{eq:recast PDE full}.
\end{proposition}
\begin{proof} 
The result follows by Subsection~\ref{ssec:ODE formulation} and Proposition~\ref{prop:well posedness nonlinear abstract}, observing that $U_0 = (u_0,v_0) \in D(\A)$.
\end{proof}

In the next result, we deduce energy estimates on the unique strong solution to~\eqref{eq:PDE}. 


\begin{proposition}[Energy balance for regular initial data]
	 \label{prop:energy balance regular initial data}
 Assume that $u_0,v_0$ satisfy~\eqref{eq:regularity initial data}. Let $(u,v)$ be the unique strong solution to~\eqref{eq:PDE} with initial data  $u_0,v_0$ obtained in Proposition~\ref{prop:well-posedness approximation with regular data}. Let $E$ and $J$ be as in~\eqref{def:E} and~\eqref{def:J}, respectively. Then the following energy balances hold true:
  \[ 
        \label{eq:approximate energy estimate}
  E(u(t),\de_t u(t)) + \int_0^t \|\de_t u(s)\|_{L^2(\Omega)}^2 \, \d s = E(u_0,v_0) \, , \quad \text{for every } t \in [0,+\infty)\, , 
  \]
  and
  \[ \label{eq:estimate on Ik}
  \begin{split}
      & J(u(t),\de_t u(t)) + \int_0^t \Big( \|\nabla u(s)\|^2_{L^2(\Omega)} + \langle u(s), \Phi'(u(s)) \rangle_{L^2(\Omega)} \Big) \, \d s \\
      & \quad = J(u_0,v_0) + \int_0^t \|\de_t u(s)\|_{L^2(\Omega)}^2 \, \d s \, , \quad \text{for every } t \in [0,+\infty) \, . 
  \end{split}
  \]
\end{proposition}
\begin{proof}
  Let us prove~\eqref{eq:approximate energy estimate}. By~\eqref{eq:recast PDE full}, by the regularity of $u$, and integrating by parts, we get that for every $t \in (0,+\infty)$ 
  \[
      \begin{split}
          & \frac{\d }{\d t}\Big(E(u(t),\de_t u(t))\Big) \\
          &  \quad  = \Big\langle \frac{\d^2 u}{\d t^2}(t) ,  \frac{\d u}{\d t}(t) \Big\rangle_{L^2(\Omega)} + \Big\langle \frac{\d \nabla u}{\d t}(t) , \nabla u(t) \Big\rangle_{L^2(\Omega)} + \Big\langle \Phi'(u(t)) , \frac{\d u}{\d t}(t) \Big\rangle_{L^2(\Omega)}  \\
          & \quad = \langle \Lap u(t) - v(t) - \Phi'(u(t)) , v(t) \rangle_{L^2(\Omega)} \\
          & \quad \quad + \langle \nabla v(t) , \nabla u(t) \rangle_{L^2(\Omega)} + \langle \Phi'(u(t)) , v(t) \rangle_{L^2(\Omega)}  = - \|v(t)\|_{L^2(\Omega)}^2 = - \Big\|\frac{\d u}{\d t}(t)\Big\|_{L^2(\Omega)}^2  \, .
      \end{split}
  \]
  Integrating in time we get~\eqref{eq:approximate energy estimate}.

  Let us prove~\eqref{eq:estimate on Ik}.  By~\eqref{eq:recast PDE full} and integrating by parts, we obtain that 
  \[
      \begin{split}
          & \frac{\d}{\d t} \Big( J(u(t),\de_t u(t)) \Big)  =  \Big\langle u(t), \frac{\d u}{\d t}(t) \Big\rangle_{L^2(\Omega)}  + \Big\|\frac{\d u}{\d t}(t)\Big\|^2_{L^2(\Omega)} +  \Big\langle u(t), \frac{\d^2 u}{\d t^2}(t) \Big\rangle_{L^2(\Omega)} \\
          & \quad =  \Big\langle u(t), \frac{\d v}{\d t}(t) + v(t) \Big\rangle_{L^2(\Omega)} + \Big\|\frac{\d u}{\d t}(t)\Big\|^2_{L^2(\Omega)} \\
          & \quad = \langle u(t), \Lap u(u(t)) - \Phi'(u(t)) \rangle_{L^2(\Omega)} + \Big\|\frac{\d u}{\d t}(t)\Big\|^2_{L^2(\Omega)} \\
          &  \quad = - \|\nabla u(t)\|_{L^2(\Omega)}^2  - \langle u(t), \Phi'(u(t)) \rangle_{L^2(\Omega)} + \Big\|\frac{\d u}{\d t}(t)\Big\|^2_{L^2(\Omega)} \, .
      \end{split}
  \]
  Integrating in time we get~\eqref{eq:estimate on Ik}.
\end{proof}

\subsection{Solution with initial data in energy space} 
\label{ssec:sol-enespace}

In this subsection we drop the higher regularity assumptions~\eqref{eq:regularity initial data}. 

\begin{proof}[Proof of Theorem \ref{prop:well-posedness approximation initial data in energy space}]
	We split the proof into several steps. 
	
	\begin{steps}
		\item (Existence and uniqueness of mild solutions)   Existence and uniqueness are obtained by Proposition~\ref{prop:solution less regular data}, which guarantees that $U = (u,v) \in C([0,+\infty);H)$ is the unique curve satisfying 
  \[ \label{eq:integral formulation Uk}
    U(t) = S_\A(t) U_0 - \int _0^t S_\A(t-s) F(U(s)) \, \d s \, , \quad t \in [0,+\infty)\, .
    \]
  
  \item (Weak solution) We need to show that $u$ is a weak solution in the sense given in the statement.  In Subsection~\ref{ssec:ODE formulation} we have shown that the operator $\A$ defined in~\eqref{eq:L} is maximal monotone. As a consequence, its domain $D(\A)$ is dense in $H=H^1(\Omega) \x L^2(\Omega)$, see~\cite[Proposition~I.1.1.2]{Har}. Hence, there exists a sequence $U_0^j = (u_0^j,v_0^j) \in D(\A)$ such that $U_0^j =  (u_0^j,v_0^j) \to U_0 = (u_0, v_0)$ in $H$ as $j \to +\infty$. Let $U^j = (u^j, v^j)$ be the unique solution to~\eqref{eq:recast PDE full} with initial datum $U^j(0) = U^j_0$, in the sense of Proposition~\ref{prop:well posedness nonlinear abstract}, satisfying 
  \[ \label{eq:integral formulation Ukj}
    U^j(t) = S_\A(t) U^j_0 - \int _0^t S_\A(t-s) F(U^j(s)) \, \d s \, , \quad t \in [0,+\infty)\, .
    \]
  Notice that, by Proposition~\ref{prop:well posedness nonlinear abstract}, $U^j =(u^j,v^j)\in C^1([0,+\infty);H) \cap C([0,+\infty); D(\A))$ and it satisfies~\eqref{eq:recast PDE compact} in a strong sense. This is the unique solution provided by Proposition~\ref{prop:well-posedness approximation with regular data}.
    
    Let us show that $U^j = (u^j,v^j)$ is a Cauchy sequence in $C([0,T];H)$ with respect to $j$ for every $T > 0$. To do so, we exploit~\eqref{eq:integral formulation Ukj}, the contraction property of $\{S_\A(t)\}_{t \geq 0}$, and~\eqref{eq:F lipschitz} to estimate for $i$ and $j$ and for every $t \in (0,+\infty)$ 
    \[
      \begin{split}
        \|U^i(t) - U^j(t)\|_H & = \Big\| S_\A(t) ( U^i_0 - U^j_0) - \int _0^t S_\A(t-s) ( F(U^i(s)) - F(U^j(s)) ) \, \d s \Big\|_H   \\
        & \leq \| U^i_0 - U^j_0 \|_H + \int_0^t (1+\Lip(\Phi'))\|U^i(s) - U^j(s)\|_H \, \d s  \, .
      \end{split}
    \]
    By Gr\"onwall's inequality, given $T > 0$, we deduce that for every $t \in [0,T]$
    \[
      \|U^i(t) - U^j(t)\|_H \leq \|U^i_0 - U^j_0\|_H e^{(1+\Lip(\Phi'))t} \leq \|U^i_0 - U^j_0\|_H^2 e^{(1+\Lip(\Phi'))T} \, ,
    \]
Since $U^j_0$ converges, we conclude that $U^j$ is a Cauchy sequence with respect to $j$ in $C([0,T];H)$. Since~\eqref{eq:integral formulation Uk} is stable with respect to the uniform convergence, we conclude that  $U^j = (u^j,v^j) \to U = (u,v)$ in $C([0,T];H)$ as $j \to +\infty$ for every $T > 0$. Furthermore, since $u^j \in C^1([0,+\infty);H^1(\Omega))$ and $\de_t u^j  = v^j \to v$ in $C([0,T];L^2(\Omega))$ as $j \to + \infty$ for every $T > 0$, we deduce that $u \in C^1([0,+\infty);L^2(\Omega))$ with derivative $\de_t u = v$ in $L^2(\Omega)$. Hence $u \in C([0,+\infty);H^1(\Omega)) \cap C^1([0,+\infty);L^2(\Omega))$ and~\eqref{eq:ODE in L2} is satisfied.

Let us prove Remark~\ref{rem:weak solution}. We multiply by $\varphi \in C^\infty_c(\R \x \R^d)$ the equation~\eqref{eq:recast PDE full} solved by $U^j = (u^j,v^j)$ in a strong sense and integrate by parts to obtain that for every $t \in (0,+\infty)$
\[
  \begin{split}
  0 & = \Big\langle \frac{\d^2 U^j}{\d t^2}(t), \varphi(t) \Big\rangle_{L^2(\Omega)} \!\! - \langle \Lap U^j(t), \varphi(t) \rangle_{L^2(\Omega)} + \Big\langle \frac{\d U^j}{\d t}(t), \varphi(t) \Big\rangle_{L^2(\Omega)} \!\! + \langle \Phi'(U^j(t)) , \varphi(t) \rangle_{L^2(\Omega)}  \\
  & = \frac{\d}{\d t} \Big\langle \frac{\d U^j}{\d t}(t), \varphi(t) \Big\rangle_{L^2(\Omega)} -  \Big\langle \frac{\d U^j}{\d t}(t),  \frac{\d \varphi}{\d t}(t) \Big\rangle_{L^2(\Omega)} + \langle \nabla U^j(t), \nabla \varphi(t) \rangle_{L^2(\Omega)} \\
  & \quad + \Big\langle \frac{\d U^j}{\d t}(t), \varphi(t) \Big\rangle_{L^2(\Omega)} + \langle \Phi'(U^j(t)) , \varphi(t) \rangle_{L^2(\Omega)} \, .
  \end{split}
\]
    Integrating in time, we get that
    \[
      \begin{split}
      0 & = \langle v_0, \varphi(0) \rangle_{L^2(\Omega)} +  \int_0^{+\infty}\Big\langle v^j(t), - \frac{\d \varphi}{\d t}(t) + \varphi(t)\Big\rangle_{L^2(\Omega)}  \, \d t + \int_0^{+\infty} \langle \nabla U^j(t), \nabla \varphi(t) \rangle_{L^2(\Omega)} \\
      & \quad + \int_0^{+\infty} \langle \Phi'(U^j(t)) , \varphi(t) \rangle_{L^2(\Omega)} \, \d t \, .
      \end{split}
    \]
    Letting $j \to +\infty$, we conclude that
    \[
      \begin{split}
      0 & = \langle v_0, \varphi(0) \rangle_{L^2(\Omega)} +  \int_0^{+\infty}\Big\langle v(t), - \frac{\d \varphi}{\d t}(t) + \varphi(t)\Big\rangle_{L^2(\Omega)}  \, \d t + \int_0^{+\infty} \langle \nabla u(t), \nabla \varphi(t) \rangle_{L^2(\Omega)} \\
      & \quad + \int_0^{+\infty} \langle \Phi'(u(t)) , \varphi(t) \rangle_{L^2(\Omega)} \, \d t \, .
      \end{split}
    \]
    This concludes the proof.
\end{steps} 
\end{proof}

Now we turn to the proof of the energy balance. 

\begin{proof}[Proof of Theorem \ref{th:energy}]
The proof of the energy balance is based on the approximation used in the proof of Theorem~\ref{prop:well-posedness approximation initial data in energy space}. Let $U^j_0 = (u^j_0,v^j_0) \to U_0 = (u_0,v_0)$ in $H$ and let $U^j = (u^j, v^j)$ the corresponding solutions obtained therein satisfying  $U^j = (u^j, v^j) \to U = (u,v)$ in $C([0,T];H)$ as $j \to +\infty$ for every $T > 0$. By Proposition~\ref{prop:energy balance regular initial data}  we have that
  \[
    E(U^j(t),\de_t U^j(t)) + \int_0^t  \|\de_t U^j (s) \|_{L^2(\Omega)}^2 \, \d s  = E(u^j_0,v^j_0) \, ,\quad \text{for every } t \in [0,+\infty) \, ,
\] 
and 
\[
  \begin{split}
      & J(U^j(t),\de_t U^j(t)) + \int_0^t \Big( \|\nabla U^j(s)\|^2_{L^2(\Omega)} + \langle U^j(s), \Phi'(U^j(s)) \rangle_{L^2(\Omega)} \Big) \, \d s \\
      & \quad = J(u^j_0,v^j_0) + \int_0^t \|\de_t U^j(s)\|_{L^2(\Omega)}^2 \, \d s \, , \quad \text{for every } t \in [0,+\infty) \, .     
  \end{split}
\]
Exploiting the convergences $U^j_0 = (u^j_0,v^j_0) \to U_0 = (u_0,v_0)$ in $H$, $u^j \to u$ in $C([0,T];H^1(\Omega))$, and $\de_t u^j  \to \de_t u$ in $C([0,T];L^2(\Omega))$, we pass to the limit as $j \to +\infty$ to get~\eqref{eq:approximate energy estimate II} and~\eqref{eq:estimate on Ik II}.

\end{proof}

\section{Proof of the long-time asymptotics} 
\label{sec:decay-approx}

In this section, we study the long-time limit of the unique solution obtained in Theorem~\ref{prop:well-posedness approximation initial data in energy space}. 

\subsection{Qualitative result}
\label{ssec:qualitative}

In Theorem~\ref{prop:long-time limit}, we claimed that the limit $\lim_{t \to +\infty} u(t)$ exists and is a solution to the stationary problem. The proof requires several intermediate results. 

Before discussion the case of the semilinear problem, we recall the asymptotic behavior of solutions to the linear homogeneous PDE
\[   \label{eq:linear homogeneous}
    \begin{cases}
        \de^2_{tt} u - \Lap u + \de_t u  + u = 0 \, , & (t,x) \in (0,+\infty) \x \Omega   \, , \\
        \nabla u \cdot \nu = 0  \, , & (t,x) \in  (0,+\infty) \x \de \Omega \, , \\
        u(0,x) = u_0(x) \, , \  \de_t u(0,x) = v_0(x) \, , &  x \in \Omega   \, .
    \end{cases}
\]   

\begin{proposition}[Long-time asymptotics for the linear homogeneous problem] \label{prop:asymptotic linear homogeneous}
  Let $\{S_\A(t)\}_{t \geq 0}$ be the continuous semigroup of contractions generated by $\A$, see Proposition~\ref{prop:semigroup of contractions}. Then there exists a constant $c > 0$ such that 
  \[
  \|S_\A(t)\|_{\L(H)} \leq c \,  e^{-t/2} \, ,  
  \]
  where $\|S_\A(t)\|_{\L(H)}$ denotes the operator norm of $S_\A(t) \colon H \to H$. 
\end{proposition}

The previous result states in an abstract form the following fact: If $u_0, v_0$ satisfy~\eqref{eq:initial data in energy space} and $(u(t),v(t))$ is the unique solution to~\eqref{eq:linear homogeneous}, then $\|u(t)\|_{H^1(\Omega)} \to 0$ and $\|\de_t u(t)\|_{L^2(\Omega)} \to 0$ exponentially fast. The main result in this paper is about the analogous result for the semilinear problem under investigation. Before going on with the discussion in the semilinear case, we provide some details for the proof of Proposition~\ref{prop:asymptotic linear homogeneous}, which follows the lines of~\cite[Lemma~9.5.1]{CazHar}.

\begin{proof}[Proof of Proposition~\ref{prop:asymptotic linear homogeneous}]
  We consider the following energy functional
  \[
    G(u,v) = \frac{1}{2}\|v\|_{L^2(\Omega)}^2 + \frac{1}{2}\|\nabla u\|_{L^2(\Omega)}^2 + \frac{1}{2}\|u\|_{L^2(\Omega)}^2 +  \frac{1}{2} \langle u,v \rangle_{L^2(\Omega)} \, .
  \]
  We observe that 
  \[
    \frac{1}{2} | \langle u,v \rangle_{L^2(\Omega)}|  \leq   \frac{1}{4} \|v\|_{L^2(\Omega)}^2 + \frac{1}{4} \|u\|_{L^2(\Omega)}^2 \, , 
  \] 
  hence
  \[ \label{eq:F bounds norms}  
   \frac{1}{4}  \|u\|_{H^1(\Omega)}^2 + \frac{1}{4}  \|v\|_{L^2(\Omega)}^2  \leq G(u,v) \leq \frac{3}{4}  \|u\|_{H^1(\Omega)}^2 + \frac{3}{4}  \|v\|_{L^2(\Omega)}^2  \, .
  \]

  Let us fix $U_0 = (u_0,v_0) \in D(\A)$ and let $U(t) = (u(t), v(t)) = S_\A(t)U_0$. By Proposition~\ref{prop:semigroup of contractions} we have that $U(t)\in C^1([0,+\infty);H) \cap C([0,+\infty); D(\A))$ solves~\eqref{eq:homogeneous ODE}, \ie, 
  \[  
\begin{dcases}
  \frac{\d u}{\d t}(t) - v(t) = 0 \, , & \text{for } t \in (0,+\infty) \, ,  \\
    \frac{\d v}{\d t}(t) - \Lap u(t) + v(t) + u(t)   =  0 \, , & \text{for } t \in (0,+\infty)  \, ,   \\
    u(0) = u_0 \, , \  v(0) = v_0 \, .
\end{dcases}
\]
Integrating by parts and exploiting the previous equation, we have that  
\[
  \begin{split}
    & \frac{\d}{\d t} \Big( G(u(t),\de_t u(t)) \Big) \\
    &  =   \frac{\d}{\d t} \Big( \frac{1}{2} \Big\|\frac{\d u}{\d t}(t)\Big\|_{L^2(\Omega)}^2 + \frac{1}{2}\|\nabla u(t)\|_{L^2(\Omega)}^2 + \frac{1}{2} \|u(t)\|_{L^2(\Omega)}^2  + \frac{1}{2}\Big\langle u(t),\frac{\d u}{\d t}(t) \Big\rangle_{L^2(\Omega)} \Big) \\
    & = \Big\langle \frac{\d^2 u}{\d t^2}(t), \frac{\d u}{\d t}(t) \Big\rangle_{L^2(\Omega)} + \Big\langle \frac{\d \nabla u}{\d t}(t), \nabla u(t) \Big\rangle_{L^2(\Omega)} + \Big\langle u(t), \frac{\d u}{\d t}(t)  \Big\rangle_{L^2(\Omega)} \\
    & \quad  + \frac{1}{2} \Big\| \frac{\d u}{\d t}(t) \Big\|_{L^2(\Omega)}^2 + \frac{1}{2} \Big\langle u(t), \frac{\d^2  u}{\d t^2}(t) \Big\rangle_{L^2(\Omega)} \\
    & = \Big\langle \Lap u(t) - \frac{\d u}{\d t}(t) - u(t) , \frac{\d u}{\d t}(t) \Big\rangle_{L^2(\Omega)} - \Big\langle \frac{\d u}{\d t}(t), \Lap u(t) \Big\rangle_{L^2(\Omega)} + \Big\langle u(t), \frac{\d u}{\d t}(t)  \Big\rangle_{L^2(\Omega)} \\
    & \quad + \frac{1}{2} \Big\| \frac{\d u}{\d t}(t) \Big\|_{L^2(\Omega)}^2   + \frac{1}{2}\Big\langle u(t),\Lap u(t) - \frac{\d u}{\d t}(t) - u(t) \Big\rangle_{L^2(\Omega)} \\
    & = - \frac{1}{2} \Big\| \frac{\d u}{\d t}(t) \Big\|_{L^2(\Omega)}^2 - \frac{1}{2} \|\nabla u(t)\|_{L^2(\Omega)}^2  - \frac{1}{2} \|u(t)\|_{L^2(\Omega)}^2 - \frac{1}{2} \Big\langle  u(t), \frac{\d u}{\d t}(t)  \Big\rangle_{L^2(\Omega)}\\
    & = - G(u(t),\de_t u(t))  \, .
  \end{split}
\]
This implies, together with~\eqref{eq:F bounds norms}, that 
\[
  \begin{split}
    \frac{1}{4} \|U(t)\|_{H}^2 & = \frac{1}{4}  \|u(t)\|_{H^1(\Omega)}^2 + \frac{1}{4}  \|\de_t u(t)\|_{L^2(\Omega)}^2  \leq G(u(t),\de_t u(t)) \\
    &  = G(u_0,v_0) e^{-t}  \leq \Big( \frac{3}{4} \|u_0\|_{H^1(\Omega)}^2 + \frac{3}{4} \|v_0\|_{L^2(\Omega)}^2 \Big) e^{-t} = \frac{3}{4} \|U_0\|_H^2 e^{-t} \, ,
  \end{split}
\]
\ie 
\[
  \|S_\A(t) U_0\|_{H}^2  \leq 3 \|U_0\|_H^2 e^{-t} \, ,
\]
By the density of $D(\A)$ in $H$, we conclude the proof. 
\end{proof}

We are now in a position to start the asymptotic analysis of solutions to the semilinear problem. The first step is to deduce a bound on $u(t)$ in $H^1(\Omega)$.
 
\begin{proposition}[$H^1$-bound] \label{prop:H1 bound}
  Assume that $u_0, v_0$ satisfy~\eqref{eq:initial data in energy space}. Let $(u,v)$ be the unique solution to~\eqref{eq:PDE} obtained in Theorem~\ref{prop:well-posedness approximation initial data in energy space}. Then 
    \[ 
      \sup_{t \geq 0} \|u(t)\|_{H^1(\Omega)} < +\infty \, .
    \]
\end{proposition}

\begin{proof} We split the proof of the estimate in two steps. 
  \begin{steps}
    \item (Estimating $\|\nabla u(t)\|_{L^2(\Omega)}$) By~\eqref{eq:approximate energy estimate II} and the definition of $E$ in~\eqref{def:E}, we obtain that 
    \[ \label{eq:bound on detu and nablau}
      \frac{1}{2} \| \de_t u(t) \|^2_{L^2(\Omega)}  + \frac{1}{2} \| \nabla u(t) \|^2_{L^2(\Omega)} \leq E(u(t), \de_t u(t))\leq E(u_0, v_0) \leq E(u_0,v_0)\, ,
      \]
      where we used that $\Phi \leq \Phi$ in the last inequality. In particular, $\sup_{t \geq 0} \|\nabla u(t)\|_{L^2(\Omega)} < +\infty$. 
      \item (Estimating $\|u(t)\|_{L^2(\Omega)}$) Estimating $\|u(t)\|_{L^2(\Omega)}$ requires several steps. The first observation is that the functional $J$ defined in~\eqref{def:J} satisfies that  
      \[ \label{eq:limit of J}
      \lim_{t \to +\infty}  J(u(t),\de_t u(t)) = \ell \, , \quad \text{for some } \ell \in \R \, .
      \]
      To prove~\eqref{eq:limit of J}, we start by observing that $t \mapsto J(u(t), \de_t u(t))$ is bounded from below, since~\eqref{eq:bound on detu and nablau} yields
      \[ \label{eq:J bounded from below}
        \begin{split}
          J(u(t), \de_t u(t)) & = \frac{1}{2} \|u(t)\|_{L^2(\Omega)}^2 + \langle u(t), \de_t u(t) \rangle_{L^2(\Omega)} \\
          & = \frac{1}{2}\| u(t) +  \de_t u(t) \|_{L^2(\Omega)}^2 - \frac{1}{2} \|\de_t u(t)\|_{L^2(\Omega)}^2 \\
          & \geq - \frac{1}{2} \|\de_t u(t)\|_{L^2(\Omega)}^2  \geq - E(u_0,v_0) \, .
        \end{split}
      \] 
      Let us show that $t \mapsto J(u(t), \de_t u(t)) - \int_0^t \|\de_t u(r)\|_{L^2(\Omega)}^2 \, \d r$ is nonincreasing. By~\eqref{eq:estimate on Ik II} and property \ref{it:p5} of $\Phi$, we have that, for every $s \leq t$,
      \[ 
      \begin{split}
          & J(u(s),\de_t u(s)) - \int_0^s \|\de_t u(r)\|_{L^2(\Omega)}^2 \, \d r  \\
          & \ = J(u(t),\de_t u(t))  - \int_0^t \|\de_t u(r)\|_{L^2(\Omega)}^2 \, \d r  \\
          & \quad +  \int_s^t  \Big( \|\nabla u(r)\|^2_{L^2(\Omega)}   + \langle u(r), \Phi'(u(r)) \rangle_{L^2(\Omega)}  \Big) \, \d r   \\
          & \ \geq J(u(t),\de_t u(t))  - \int_0^t \|\de_t u(r)\|_{L^2(\Omega)}^2 \, \d r \, .
      \end{split}
      \]
      The term $\int_0^t \|\de_t u(r)\|_{L^2(\Omega)}^2 \, \d r$ admits a limit as $t \to +\infty$. Indeed, equality~\eqref{eq:approximate energy estimate II} yields 
      \[
      \sup_{t \geq 0}\int_0^t \|\de_t u(s)\|^2_{L^2(\Omega)} \, \d s  \leq E(u_0, v_0) \leq E(u_0,v_0) \, ,
      \]
      hence $\de_t u \in L^2([0,+\infty); L^2(\Omega))$ and 
      \[ \label{eq:improper integral}
      \lim_{t \to +\infty}   \int_0^t \|\de_t u(s)\|^2_{L^2(\Omega)} \, \d s = \int_0^{
        +\infty} \|\de_t u(s)\|^2_{L^2(\Omega)} \, \d s \, .
      \]
  \end{steps}
   The monotonicity of $t \mapsto J(u(t), \de_t u(t))$, together with~\eqref{eq:J bounded from below} and \eqref{eq:improper integral}, gives

We exploit~\eqref{eq:limit of J} to show that 
\[ \label{eq:limit of L2 norm}
  \lim_{t \to +\infty}  \frac{1}{2} \|u(t)\|_{L^2(\Omega)}^2 =  \ell \, .
\]
By definition of~$J$ we have that 
  \[
    \begin{split}
      J(u(t), \de_t u(t)) & = \frac{1}{2} \|u(t)\|_{L^2(\Omega)}^2 + \langle u(t), \de_t u(t) \rangle_{L^2(\Omega)} \\
      & = \frac{1}{2} \|u(t)\|_{L^2(\Omega)}^2 + \frac{\d}{\d t} \Big( \frac{1}{2} \|u(t)\|_{L^2(\Omega)}^2 \Big) \, ,
    \end{split}
  \]
  \ie, the function $t \mapsto \frac{1}{2} \|u(t)\|_{L^2(\Omega)}^2$ solves the Cauchy problem
  \[
    \begin{dcases}
      \frac{\d}{\d t} \Big( \frac{1}{2} \|u(t)\|_{L^2(\Omega)}^2  \Big) + \frac{1}{2} \|u(t)\|_{L^2(\Omega)}^2 = J(u(t), \de_t u(t)) \, , &  t \in (0,+\infty) \, ,  \\
      \frac{1}{2} \|u(0)\|_{L^2(\Omega)}^2 = \frac{1}{2} \| u_0 \|_{L^2(\Omega)}^2 \, ,
    \end{dcases}
  \]
  hence
  \[
    \frac{1}{2} \|u(t)\|_{L^2(\Omega)}^2 =  \frac{1}{2} \| u_0 \|_{L^2(\Omega)}^2 e^{-t} + \int_0^t e^{s-t} J(u(s), \de_t u(s)) \, \d s \, .
  \]
  We observe that, by~\eqref{eq:limit of J} and l'H\^{o}pital's rule 
  \[
    \begin{split}
      \ell & = \lim_{t \to +\infty} J(u(t), \de_t u(t)) =  \lim_{t \to +\infty}  \frac{   e^{t} J(u(t), \de_t u(t))}{e^t} \\
      & =  \lim_{t \to +\infty}  \frac{ \int_0^t e^{s} J(u(s), \de_t u(s)) \, \d s}{e^t} = \lim_{t \to +\infty}  \int_0^t e^{s-t} J(u(s), \de_t u(s)) \, \d s \, ,
    \end{split}
  \]
  thus~\eqref{eq:limit of L2 norm} follows. In particular, $\sup_{t \geq 0}\|u(t)\|_{L^2(\Omega)} < +\infty$. This concludes the proof. 
  
\end{proof}

Next, we study the accumulation points of $u(t)$  as $t \to +\infty$. 

\begin{proposition}[$\omega$-limit set]  \label{prop:omega-limit set}
  Assume that $u_0, v_0$ satisfy~\eqref{eq:initial data in energy space}. Let $(u,v)$ be the unique solution to~\eqref{eq:PDE} obtained in Theorem~\ref{prop:well-posedness approximation initial data in energy space}. For every sequence $(t_n)_{n \in \N}$ with $t_n \nearrow+\infty$ there exists a subsequence $(t_{n_m})_{m\in \N}$ and $u_\infty \in H^1(\Omega), v_\infty  \in L^2(\Omega)$ such that $u(t_{n_m}) \to u_\infty$  strongly in $H^1(\Omega)$ and $\de_t u(t_{n_m}) \to v_\infty$ strongly in $L^2(\Omega)$ as $m \to +\infty$. Moreover, $v_\infty = 0$ and $u_\infty$ is a weak solution to the stationary problem
  \[  \label{eq:stationary PDE}
    \begin{cases}
         - \Lap u_\infty + \Phi'(u_\infty) = 0 \, , & x \in  \Omega   \, , \\
        \nabla u_\infty \cdot \nu = 0  \, , & x \in    \de \Omega \,  .
    \end{cases}
\]
\end{proposition}
\begin{proof} We divide the proof in two steps.
\begin{steps}
  \item (Compactness) The proof follows the lines of~\cite[Lemma~9.5.2]{CazHar}. We provide the details for the sake of completeness. 

  By Theorem~\ref{prop:well-posedness approximation initial data in energy space}, we have that $u = (u,v) \in C([0,+\infty); H)$ solves
  \[ \label{eq:Uk solves}
    u(t) = S_\A(t) U_0 - \int _0^t S_\A(t-s) F(u(s)) \, \d s \, , \quad t \in [0,+\infty)\, ,  
  \]
  where $U_0 = (u_0,v_0) \in H$, $F$ is defined in Remark~\ref{rmk:maximal monotone}, and $\{S_\A(t)\}_{t \geq 0}$ is the continuous semigroup of contractions associated to $\A$ (see Proposition~\ref{prop:semigroup of contractions}). By Proposition~\ref{prop:asymptotic linear homogeneous}, we have that $\|S_\A(t) U_0 \|_H \to 0$, hence 
  \[ \label{eq:compact 1}
    \{S_\A(t) U_0\}_{t \geq 0} \subset K_1 \, , \quad \text{for some compact set } K_1 \subset H \, .
  \]
  
  To conclude, it is enough to show that the family $\{ \int _0^t S_\A(t-s) F(u(s)) \, \d s\}_{t \geq 0}$ is contained in a compact set of $H^1(\Omega)$. Let us set 
  \[
  V(t) := \int_0^t  S_\A(t-s) F(u(s)) \, \d s = \int_0^t  S_\A(s) F(u(t-s)) \, \d s \, . 
  \]
  Let us fix $\e > 0$. By Proposition~\ref{prop:H1 bound}, we have, in particular, that $\sup_{t \geq 0} \|u(t)\|_{L^2(\Omega)} < + \infty$. Recalling the boundedness of $\Phi'$, we deduce that 
  \[
    M:= \sup_{t \geq 0}\|F(u(t))\|_H \leq \sup_{t \geq 0} ( \|u(t)\|_{L^2(\Omega)} + \|\Phi'\|_\infty |\Omega|^\frac{1}{2}) < + \infty \, .
    \] 
  By the previous bound and by Proposition~\ref{prop:asymptotic linear homogeneous}, there exists $T > 0$ such that 
  \[ \label{eq:tail of integral}
  \int_T^{+\infty}  \| S_\A(s) \|_{\L(H)} \| F(u(t-s)) \|_{H}\, \d s   \leq  \int_T^{+\infty}  c \, e^{-s/2}  M \, \d s  < \e \, .
  \]

  The curve $t \mapsto V(t) \in H$ is continuous, hence it maps the interval $[0,T]$ into a compact set. This implies that the family $\{V(t)\}_{0\leq t \leq T}$ is compact and, as such,
  \[ \label{eq:V is totally bounded}
    \{V(t)\}_{0\leq t \leq T} \text{ is totally bounded.}
  \]
  
  Let us study the family $\{V(t)\}_{t \geq T}$. Thanks to~\eqref{eq:tail of integral}, we approximate $V(t)$ with 
  \[
  \tilde V(t) :=  \int_0^T   S_\A(s) F(u(t-s)) \, \d s \, .
  \]
  Indeed, by~\eqref{eq:tail of integral}, for $t \geq T$ we have that 
  \[ \label{eq:V and tilde V}
    \begin{split}
      &  \| V(t) - \tilde V(t)  \|_H   \leq \int_T^t \| S_\A(s) \|_{\L(H)} \| F(u(t-s)) \|_{H} \, \d s < \e/2 \, .
    \end{split}
  \] 
  Let us show that the family $\{\tilde V(t)\}_{t \geq T}$ is totally bounded. We start by observing that the nonlinearity $F \colon H \to H$ maps sets of the form $B \x L^2(\Omega)$, with $B \subset H^1(\Omega)$ bounded, into relatively compact sets, thanks to the compact embedding $H^1(\Omega) \subset \subset L^2(\Omega)$ and the specific expression of $F$ in~\eqref{eq:Fk}. Hence, there exists a compact set $K \subset H$ such that $\{F(u(t))\}_{t \geq 0} \subset K$. It follows that the family $\{S_\A(s)F(u(t-s))\}_{0 \leq s \leq T}$ is contained in the set $\bigcup_{0\leq s \leq T} S_\A(s)K$, which is compact, since it is given by the image of the compact set $[0,T] \x K$ through the jointly continuous map $(t,U)\in [0,+\infty) \x H \mapsto S_\A(t)U \in H$. The operator $I \colon C([0,T];H) \to H$ given by $I(U) = \int_0^T U(s) \, \d s$ is continuous, thus it maps the set $\bigcup_{0\leq s \leq T} S_\A(s)K$ into a compact set $K_2 \subset H$. We conclude that  $\{\tilde V(t)\}_{t \geq T} \subset K_2$, hence  $\{\tilde V(t)\}_{t \geq T}$ is totally bounded. In particular,  $\{\tilde V(t)\}_{t \geq T}$ can be covered by finitely-many balls with radii $\e/2$. By~\eqref{eq:V and tilde V}, we infer that $\{V(t)\}_{t \geq T}$ can be covered by finitely-many balls with radii $\e$. Together with~\eqref{eq:compact 1} and~\eqref{eq:V is totally bounded}, this shows that the family $\{u(t)\}_{t \geq 0}$ is totally bounded. In particular, we deduce that for every sequence $t_n \nearrow +\infty$ we can extract a subsequence (that we do not relabel) $t_{n}$ such that 
  \[ \label{eq:convergence to Uinfty}
    U(t_{n}) = (u(t_{n}), \de_t u(t_{n}))\to U_\infty = (u_\infty, v_\infty) \quad \text{in }H \text{ as } n \to +\infty \, .
  \]
  This concludes the proof of compactness.   

  \item (Stationary problem) We use $U_\infty = (u_\infty, v_\infty)$ as initial datum for the PDE, \ie, we consider the unique solution $w(t)$ to 
  \[   
    \begin{cases}
        \de^2_{tt} w - \Lap w + \de_t w + \Phi'(w) = 0 \, , & \text{in } (0,+\infty) \x \Omega   \, , \\
        \nabla w \cdot \nu = 0  \, , & \text{on }  (0,+\infty) \x \de \Omega \, , \\
        w(0,\cdot) = u_\infty \, , \  \de_t w(0,\cdot) = v_\infty \, , & \text{in } \Omega   \, .
    \end{cases}
\]  
  More rigorously, $W(t) = (w(t),\de_t w(t)) \in C([0,+\infty); H)$ is the unique curve satisfying 
  \[
    W(t) = S_\A(t)U_\infty - \int_0^t S_\A(s)F(W(t-s)) \, \d s \, , 
  \]
  provided by Theorem~\ref{prop:well-posedness approximation initial data in energy space}. 

  By Remark~\ref{rmk:dynamical system operator}, there exists a family of operators $\{T(t)\}_{t \geq 0}$ satisfying Properties~\ref{it:1}--\ref{it:2} and such that $T(t)U_0$ is the unique solution to the evolution problem with initial datum $U_0$. Adopting this operator, we have that the solution $u(t)$ considered in the previous step and satisfying~\eqref{eq:Uk solves} is given by $u(t) = T(t) (U_0)$. Analogously, $W(t) = T(t) (U_\infty)$.
  
  In the previous step we have shown that $T(t_n)(U_0) = u(t_n) \to U_\infty$ in $H$. By continuity of $T(t)$, we get that
  \[ \label{eq:convergence t plus tn}
    u(t+t_n) = T(t+t_n)(U_0) = T(t)\circ T(t_n)(U_0) \to T(t)(U_\infty) = W(t) \, , \quad \text{ in } H \text{ as } n \to +\infty \, . 
  \]
  Let us now consider the energy functional $E$ defined in~\eqref{def:E}. By Theorem \ref{th:energy}, we have that for every $s \leq t$
  \[
    E(u(t),\de_t u(t)) + \int_s^t \|\de_t u(r)\|_{L^2(\Omega)}^2 \, \d r = E(u(s),\de_t u(s))  \, ,
  \]
  and, in particular, that $t \mapsto E(u(t),\de_t u(t))$ is nonincreasing. Being bounded from below, we deduce the existence of $e_\infty \geq 0$ such that  $\lim_{t \to +\infty}E(u(t),\de_t u(t)) = e_\infty$. The functional $E$ is continuous with respect to the convergence $u(t_{n}) = (u(t_{n}), \de_t u(t_{n}))\to U_\infty = (u_\infty, v_\infty)$ in $H$ as $n \to +\infty$, hence 
  \[
    e_\infty  = \lim_{n \to +\infty}   E(u(t_n),\de_t u(t_n)) =  E(u_\infty,v_\infty)\, .
  \]
  Analogously, by~\eqref{eq:convergence t plus tn} we have that 
  \[
   e_\infty = \lim_{n \to +\infty}   E(u(t+t_n),\de_t u(t+t_n)) = E(w(t),\de_t w(t)) \, .
  \]
  The energy balance obtained in Theorem \ref{th:energy} yields that 
  \[
  e_\infty +\int_0^t \|\de_t w(s)\|_{L^2(\Omega)}^2 \, \d s =  E(w(t),\de_t w(t))  + \int_0^t \|\de_t w(s)\|_{L^2(\Omega)}^2 \, \d s = E(u_\infty,v_\infty) = e_\infty \, ,
  \]
  for every $t \in [0,+\infty)$, therefore 
  \[
  \int_0^t  \|\de_t w(s)\|_{L^2(\Omega)}^2 \, \d s = 0 \, ,
  \]
  for every $t \in [0,+\infty)$. This implies that $\de_t w(t) = 0$ for every $t \in (0,+\infty)$, \ie, $w(t)$ is constant in time. In particular, $v_\infty = 0$ and $u_\infty$ is a weak solution to the stationary problem~\eqref{eq:stationary PDE}.
\end{steps}
\end{proof}

In the next proposition, we characterize the solutions to the stationary problem. 
\begin{proposition}[Solutions to the stationary problem] \label{prop:stationary problem}
  Let $u_\infty \in H^1(\Omega)$ be a weak solution to the stationary problem \eqref{eq:stationary PDE}. 
Then $u_\infty$ is constant a.e.\ in $\Omega$ and $u_\infty \in \{\Phi' = 0\}$.
\end{proposition}
\begin{proof}
  We test the equation with $u_\infty$. We get that 
  \[
    \int_{\Omega} |\nabla u_\infty|^2 \, \d x + \int_{\Omega} u_\infty \Phi'(u_\infty) \, \d x = 0 \, .
  \]
  By property \ref{it:p5} of $\Phi$, we have that $u_\infty \Phi'(u_\infty) \geq 0$, hence both terms must equal zero. By the connectedness of $\Omega$ we deduce that $u_\infty$ is constant a.e.\ in $\Omega$. From the condition $u_\infty \Phi'(u_\infty)$ we infer that either $u_\infty = 0$ or $u_\infty \in \{\Phi' = 0\}$. Note that $0 \in \{\Phi' = 0\}$. This concludes the proof. 
\end{proof}

The previous result does not guarantee uniqueness of the accumulation points of $u(t)$ for large $t$, since there are infinitely many values that annihilate $\Phi'$. To obtain this result and therefore existence of the limit for large $t$ (as claimed in Theorem \ref{prop:long-time limit}),  we exploit the energy balance~\eqref{eq:estimate on Ik II}.

\begin{proof}[Proof of Theorem \ref{prop:long-time limit}]
  The fact that 
  \[ \label{eq:detu to zero}
  \de_t u(t) \to 0 \quad  \text{ strongly in } L^2(\Omega) \text{ for } t \to +\infty
  \]
  follows from the equality $v_\infty = 0$ obtained in Proposition~\ref{prop:omega-limit set}. 

  Let us now fix two sequences $t_n \nearrow +\infty$ and $t_m \nearrow +\infty$ and $u_\infty, u_\infty' \in H^1(\Omega)$ such that $u(t_n) \to u_\infty$ strongly in $H^1(\Omega)$ as $n \to +\infty$ and $u(t_m) \to u_\infty'$ strongly in $H^1(\Omega)$ as $m \to +\infty$. By Proposition~\ref{prop:stationary problem}, $u_\infty$ and $u_\infty'$ are a.e.\ constant in $\Omega$ and satisfy $u_\infty, u_\infty' \in \{\Phi' = 0\}$.
  
  On the one hand, as a consequence of the energy balance in  Theorem~\ref{prop:well-posedness approximation initial data in energy space}, we have that (see the proof of Proposition~\ref{prop:H1 bound} for the details) the following limit exists:  
  \[
  \lim_{t \to +\infty}  J(u(t), \de_t u(t)) = \ell \, . 
  \]
  By definition~\eqref{def:J}, $\frac{1}{2} \|u(t)\|_{L^2(\Omega)}^2 + \langle u(t), \de_t u(t) \rangle_{L^2(\Omega)} \to \ell$ as $t \to +\infty$. By the convergence~\eqref{eq:detu to zero} and the bound obtained in Proposition~\ref{prop:H1 bound}, we get that $\langle u(t), \de_t u(t) \rangle_{L^2(\Omega)} \to 0$ as $t \to +\infty$. It follows that $\|u(t)\|_{L^2(\Omega)} \to \sqrt{2\ell}$. 

  On the other hand, the convergences $u(t_n) \to u_\infty$ and $u(t_m) \to u_\infty'$ yield $\|u(t_n)\|_{L^2(\Omega)} \to \|u_\infty\|_{L^2(\Omega)}$ and $\|u(t_m)\|_{L^2(\Omega)} \to \|u_\infty'\|_{L^2(\Omega)}$, therefore 
  \[ \label{eq:modulus independent}
  |u_\infty| = \sqrt{2 \ell}\, |\Omega|^{-\frac{1}{2}}   =  |u_\infty'| \, ,
  \]
  \ie, the modulus of the limit does not depend on the sequence of times. 
  
  If $\ell = 0$, then $|u_\infty| = |u_\infty'| = 0$ independently of the subsequence. Thus $\|u(t)\|_{H^1(\Omega)} \to 0$ as $t \to +\infty$.

  If $\ell > 0$, let us show that $u_\infty = u_\infty'$ arguing by contradiction. Let us assume that $u_\infty \neq u_\infty'$, \ie, $u_\infty' = - u_\infty$. Let $\delta > 0$ be such that the balls $B_\delta(u_\infty), B_\delta(u_\infty') \subset H^1(\Omega)$ of radius $\delta$ centered in $u_\infty$ and $u_\infty'$, respectively, are disjoint, \ie, $B_\delta(u_\infty) \cap B_\delta(u_\infty') = \emptyset$. By suitably extracting subsequences from $(t_n)_n$ and $(t_m)_m$, we construct a new sequence $(s_j)_j$, $s_j \nearrow +\infty$, such that $u(s_j) \in B_\delta(u_\infty)$  for  $j$  even and $u(s_j) \in B_\delta(u_\infty')$ for $j$ odd. Let us consider two consecutive times $s_j$ and $s_{j+1}$, assuming, without loss of generality, that $j$ is even. By the connectedness of the image of the continuous curve $u \colon [s_j, s_{j+1}] \to H^1(\Omega)$, there exists a time labelled $s_{j+1/2} \in (s_j, s_{j+1})$ such that $u(s_{j+1/2}) \notin  B_\delta(u_\infty) \cup B_\delta(u_\infty')$, \ie, 
  \[ \label{eq:distant from limits}
  \|u(s_{j+1/2}) - u_\infty\|_{H^1(\Omega)}  \geq \delta \, , \quad  \|u(s_{j+1/2}) - u_\infty'\|_{H^1(\Omega)}  \geq \delta \, .
  \]
  By Proposition~\ref{prop:omega-limit set} we find $u_\infty''$ and we extract a subsequence (not relabeled) from $s_{j+1/2}$ such that $u(s_{j+1/2}) \to u_\infty''$. By Proposition~\ref{prop:stationary problem}, $u_\infty''$ is a.e.\ constant in $\Omega$. Passing to the limit in~\eqref{eq:distant from limits}, we obtain that 
  \[
    |u_\infty'' - u_\infty|   \geq \delta |\Omega|^{-\frac{1}{2}}\, , \quad  |u_\infty'' + u_\infty| = |u_\infty'' - u_\infty'|   \geq \delta |\Omega|^{-\frac{1}{2}} \, ,
  \]
  and, in particular, $|u_\infty''| \neq |u_\infty|$. However, in~\eqref{eq:modulus independent} we have shown that the modulus of the limit is independent of the sequence of times. This is a contraction, thus necessarily $u_\infty = u_\infty'$.
\end{proof}

\section{Proof of the exponential decay rate}
\label{sec:rate-approx}

We now turn to the proof of the convergence rate in Theorem \ref{prop:decay-rate}. We stress that we work in dimension $d = 1$ and assume that $\Omega$ is the interval $\Omega = (0,L) \subset \R$.

Inspired by~\cite{CazHar, HarZua88}, we exploit an energy balance for the auxiliary functionals $G_\lambda \colon H^1(\Omega) \x L^2(\Omega) \to [0,+\infty)$ defined~by 
\[ \label{def:Gklambda}
    G_\lambda(u,v) = \frac{1}{2}\|v\|_{L^2(\Omega)}^2 + \frac{1}{2}\|\nabla u\|_{L^2(\Omega)}^2  + \int_{\Omega}( \Phi(u) - \Phi(u_\infty)) \, \d x + \lambda \langle u - u_\infty, v \rangle_{L^2(\Omega)} \, ,
\]
where $\lambda \in (0,1)$ will be suitably chosen later and $u_\infty \in \{\Phi' = 0\}$. This is a perturbation (depending on $\lambda$)  of $E(u,v) - E(u_\infty,0)$.

\begin{proof}[Proof of Theorem \ref{prop:decay-rate}] We split the proof in several steps.
\begin{steps} 
  \item (Energy estimate for $G_\lambda$ with regular initial data) In this step, we assume that the initial data $u_0, v_0$ are more regular and satisfy~\eqref{eq:regularity initial data}. We will relax this assumption in the next step. By Proposition~\ref{prop:well-posedness approximation with regular data}, this guarantees higher regularity for the solution, \ie, $(u,v) \in C^1([0,+\infty);H) \cap C([0,+\infty); D(\A))$, which justifies the following computations: 
\[
  \begin{split}
    G_\lambda(u(t), \de_t u(t)) & = \frac{1}{2}\|\de_t u(t)\|_{L^2(\Omega)}^2 + \frac{1}{2}\|\nabla u(t)\|_{L^2(\Omega)}^2  + \int_{\Omega} ( \Phi(u(t)) - \Phi(u_\infty) ) \, \d x \\
    & \quad  + \lambda \langle u(t) - u_\infty, \de_t u(t) \rangle_{L^2(\Omega)}  \, .
  \end{split}
\]
Recalling that $u_\infty$ is a.e.\ constant, exploiting the PDE~\eqref{eq:PDE}, and integrating by parts, we obtain that 
 \[ \label{eq:derivative of Gklambda}
  \begin{split}
   &  \frac{\d}{\d t} \Big(G_\lambda(u(t), \de_t u(t)) \Big)  = \langle \de_t u(t), \de_{tt} u(t) \rangle_{L^2(\Omega)}  +  \langle \nabla u(t), \de_t  \nabla u(t) \rangle_{L^2(\Omega)} \\
   & \quad + \langle \Phi'(u(t)), \de_t u(t) \rangle_{L^2(\Omega)}  + \lambda \| \de_t u(t) \|_{L^2(\Omega)}^2 + \lambda \langle u(t) - u_\infty, \de_{tt} u(t) \rangle_{L^2(\Omega)}  \\
   & = \langle \de_t u(t), \Lap u(t) - \de_t u(t) - \Phi'(u(t))\rangle_{L^2(\Omega)} - \langle \Lap u(t), \de_t u(t) \rangle_{L^2(\Omega)}+ \lambda \| \de_t u(t) \|_{L^2(\Omega)}^2   \\
   & \quad + \langle \Phi'(u(t)), \de_t u(t) \rangle_{L^2(\Omega)}  + \lambda \langle u(t) - u_\infty, \Lap u(t) - \de_t u(t) - \Phi'(u(t)) \rangle_{L^2(\Omega)}  \\
   & = -(1-\lambda) \|\de_t u(t)\|_{L^2(\Omega)}^2  - \lambda \|\nabla u(t)\|_{L^2(\Omega)}^2 -  \lambda \langle u(t) - u_\infty, \Phi'(u(t)) \rangle_{L^2(\Omega)} \\
   & \quad  - \lambda \langle u(t) - u_\infty, \de_t u(t) \rangle_{L^2(\Omega)} \,, \qquad \text{ for }  t >0.
  \end{split}
 \]
 
 To reconstruct $G_\lambda$ on the right-hand side of the previous expression and thus apply Gr\"onwall's inequality, we need to compare $\int_{\Omega} ( \Phi(u(t)) - \Phi(u_\infty) ) \, \d x$ and $  \langle u(t) - u_\infty, \Phi'(u(t)) \rangle_{L^2(\Omega)}$.
 
 We claim that there exists a $T_\sigma > 0$ (possibly depending on $\sigma$) such that 
 \[ \label{eq:estimate Phi with uPhi'}
  2 \int_{\Omega} ( \Phi(u(t)) - \Phi(u_\infty) ) \, \d x \leq  \langle u(t) - u_\infty, \Phi'(u(t)) \rangle_{L^2(\Omega)} \quad \text{for } t \geq T_\sigma \, .
\]
  By the convergence $u(t) \to u_\infty$ in $H^1(\Omega)$ as $t \to +\infty$ and the compact embedding in one dimension $H^1(\Omega) \subset \subset C(\overline \Omega)$, we deduce that $u(t) \to u_\infty$ \emph{uniformly} as $t \to +\infty$.  Let us distinguish two cases:
  \begin{enumerate}[align=left, label={\itshape Case \arabic*}:] 
    \item  $u_\infty = 0$.  Then there exists $T_\sigma > 0$ such that $|u(t)| < u_* - \frac{1}{\sigma}$ for $t \geq T_\sigma$, \ie, $\Phi(u(t)) = u(t)^2$ (see the definition of $\Phi$ in~\eqref{def:Phik}). Then, observing that $u(t) \Phi'(u(t)) = 2 u(t)^2$, we deduce~\eqref{eq:estimate Phi with uPhi'}.
    \item $|u_\infty| > u_*$. Then there exists $T_\sigma > 0$ such that $|u(t)| > u_*$  for $t \geq T_\sigma$, and thus, by  the definition of $\Phi$ in~\eqref{def:Phik}, both terms in~\eqref{eq:estimate Phi with uPhi'} vanish. 
  \end{enumerate}
  
Combining~\eqref{eq:derivative of Gklambda}--\eqref{eq:estimate Phi with uPhi'}, we obtain that 
\[
  \begin{split}
  & \frac{\d}{\d t} \Big(G_\lambda(u(t), \de_t u(t)) \Big)  \leq -(1-\lambda) \|\de_t u(t)\|_{L^2(\Omega)}^2  - \lambda \|\nabla u(t)\|_{L^2(\Omega)}^2 \\
  &\quad  - 2 \lambda \int_{\Omega} ( \Phi(u(t)) - \Phi(u_\infty) ) \, \d x  - \lambda \langle u(t) - u_\infty, \de_t u(t) \rangle_{L^2(\Omega)} \\
  & \leq - \kappa_\lambda G_\lambda(u(t), \de_t u(t))\, ,\qquad  \text{ for } t >T_\sigma ,
  \end{split}
\]
where $\kappa_\lambda > 0$ is a constant depending on the choice of $\lambda \in (0,1)$. 

%
\item (Energy estimate for $G_{\lambda}$ with general initial data) By a density argument (as in the proof of Theorem \ref{prop:well-posedness approximation initial data in energy space}), we can establish the same result for data that are only in the energy space.

\item (Decay of the solution)  From the decay of $G_{\lambda}$, we now want to deduce the decay of the energy $E$ and of the solution itself. 

We split the analysis into two cases as above.

  \begin{enumerate}[align=left, label={\itshape Case \arabic*}:] 
	\item  $u_\infty = 0$.  Then there exists $T_\sigma > 0$ such that $|u(t)| < u_* - \frac{1}{\sigma}$ for $t \geq T:\sigma$, \ie, $\Phi(u(t)) = u(t)^2$. Thus, for $t > T_\sigma$, the problem reduces to a damped Klein--Gordon equation: 
	\[   \label{eq:KG}
	\begin{cases}
	\de^2_{tt} u - \Lap u + \de_t u + 2 u = 0\, , & (t,x) \in (T_\sigma,+\infty) \x \Omega   \, , \\
	\nabla u \cdot \nu = 0  \, , & (t,x) \in  (T_\sigma,+\infty) \x \de \Omega \, ,
	\end{cases}
	\]
	with initial data $u(T_\sigma)$ and $\partial_t u(T_\sigma)$ at time $t = T_\sigma$. 
	Then, we go back to the study of the evolution of $G_\lambda$:  
	\begin{align*}
	\frac{\d}{\d t} G_\lambda(u(t),\partial_t u(t)) &= - (1-\lambda) \|\partial_t u(t) \|^2_{L^2(\Omega)} - \lambda \|\nabla u(t)\|_{L^2(\Omega)} \\ &\qquad - 2 \lambda \|u(t)\|^2_{L^2(\Omega)} - \lambda \langle u(t), \partial_t u(t)\rangle.
	\end{align*}
	Noticing that $2u^2 \ge u^2$ (and thus $-2u^2 \le - u^2$), choosing $\lambda = 1/2$, and applying Gr\"onwall's inequality, we deduce 
	$$G_{1/2}(u(t),\partial_t u(t)) \le e^{-\frac{1}{2}(t-T_\sigma)} G_{1/2}(u(T_\sigma),\partial_t u(T_\sigma)).$$
	Using Young's inequality and the monotonicity of the energy, we deduce 
	$$E(u(t),\partial_t u(t))  \le C \, e^{-\frac{1}{2}(t-T_\sigma)} E(u(T_\sigma),\partial_t u(T_\sigma)) \le C\, e^{-\frac{1}{2}(t-T_\sigma)} E(u_0,u_1)\, , \qquad \text{ for $t> T_\sigma$},$$
	which, in turn, directly implies 
	$$\|\partial_t u(t)\|^2_{L^2(\Omega)} + \|u(t)\|^2_{H^1(\Omega)} \le C e^{-\frac{1}{2}(t-T_\sigma)} \eqqcolon M_\sigma e^{-t/2}\, , \qquad \text{ for $t> T_\sigma$}.$$
	\item $|u_\infty| > u_*$. Then there exists $T_\sigma > 0$  such that $|u(t)| > u_*$  for $t \geq T_\sigma$, and thus, by  the definition of $\Phi$ in~\eqref{def:Phik}, the problem reduces to a damped wave equation 
	\[   \label{eq:DW}
	\begin{cases}
	\de^2_{tt} u - \Lap u + \de_t u = 0 \, , & (t,x) \in (T_\sigma,+\infty) \x \Omega   \, , \\
	\nabla u \cdot \nu = 0  \, , & (t,x) \in  (T_\sigma,+\infty) \x \de \Omega \, ,
	\end{cases}
	\]
		with initial data $u(T_\sigma)$ and $\partial_t u(T_\sigma)$ at time $t = T_\sigma$. 
Then, going back to the evolution of $G_\lambda$, 
	\begin{align*}
	\frac{\d}{\d t} G_\lambda(u(t),\partial_t u(t))  &= - (1-\lambda) \|\partial_t u(t) \|^2_{L^2(\Omega)} - \lambda \|\nabla u(t)\|^2_{L^2(\Omega)} \\ & - \lambda \langle u(t)-u_\infty, \partial_t u(t)\rangle \, .
	\end{align*}
	Choosing $\lambda = 1/2$ and applying Gr\"onwall's inequality and the monotonicity of the energy, we deduce 
	\begin{align*} 
&	E(u(t),\partial_t u(t)) - E(u_\infty,0)  \\ &\le C e^{-\frac{1}{2}(t-T_\sigma)} \Big( E(u(T_\sigma),\partial_t u(T_\sigma)) - E(u_\infty,0) \Big) \\ &\le e^{-\frac{1}{2}(t-T_\sigma)} \Big( E(u_0,u_1) - E(u_\infty,0) \Big)\, , \qquad \text{ for $t> T_\sigma$}.
	\end{align*} 
	
	Now, in order to study the decay of the solution itself, we need to apply Poincaré's inequality,
	\begin{align*}
	c_p  \|u(t) - \bar u(t) \|^2_{L^2(\Omega)} \le \|\nabla u(t)\|^2_{L^2(\Omega)},
	\end{align*}
	which yields---when combined with Young's inequality---
	\begin{align*}
\|\partial_t u(t)\|^2_{L^2(\Omega)} +	\|u(t) - \bar u(t) \|^2_{H^1(\Omega)} \le C e^{-\frac{1}{2}(t-T_\sigma)} \eqqcolon M_\sigma e^{-t/2}\, , \qquad \text{ for $t> T_\sigma$}.
	\end{align*}
	On the other hand, the average $\bar u$ satisfies the ODE
	\begin{align}\label{eq:averageODE-DW}
	\ddot{\bar u}(t) + \dot{\bar u}(t)  = 0, \qquad t >T_\sigma, 
	\end{align}
	with initial data $\bar{u}(T_\sigma)$ and $\dot{\bar{u}}(T_\sigma)$ at $t = T_\sigma$, 
	which can be explicitly solved: 
	\begin{align*}
	\bar u(t) &= \bar u(T_\sigma) + \dot{\bar u}(T_\sigma) -  e^{-(t-T_\sigma)} \dot{\bar u}(T_\sigma)
	\end{align*}
	
	In conclusion, we have 
	\begin{align*}
\|\partial_t u(t)\|_{L^2(\Omega)}^2 +	\|u(t) - (\bar u(T_\sigma) + \dot{\bar u}(T_\sigma))\|^2_{H^1(\Omega)} \le \widetilde M_\sigma e^{-\kappa t} \, , \qquad \text{ for $t> T_\sigma$}.
	\end{align*}
	Since 
	\begin{align*}
	\|u(t) - u_\infty \|_{H^1(\Omega)} \to 0 \qquad \text{ as } t \to + \infty, 
	\end{align*} 	
	we have that $u_\infty \equiv (\bar u(T_\sigma) + \dot{\bar u}(T_\sigma))$. We stress that, integrating \eqref{eq:averageODE-DW}, we have that $\dot{u}(t) + u(t)$ is constant for $t > T_\sigma$; hence, the limit does not depend on the choice of the threshold time  $T_\sigma$.	
\end{enumerate}

\end{steps}

\end{proof}

\appendix

\section{Analysis of the  ODE model}
\label{app:ODE}

In the following, we assume---without loss of generality---that  $u_* = 1$.  For every $\sigma \geq 1$ we let $\Phi$ be as in~\eqref{def:Phik} and we consider the scalar ODE:
\[ \label{eq:ODE}
\begin{cases}
  \ddot z(t) + \dot z(t) + \Phi'(z(t)) = 0 \, , \quad t \in (0,+\infty) \, ,\\
  z(0) = z_0 \, , \\
  \dot z(0) = w_0 \, ,
\end{cases}  
\]
modelling a damped nonlinear spring. We drop the dependence on $\sigma$ not to overburden the notation, but we invite the reader to pay attention to its role. 

By the Cauchy--Lipschitz theorem, there exists a unique $C^1$ solution to~\eqref{eq:ODE}. We can study the long-time behavior of the trajectories of \eqref{eq:ODE} and, in particular, prove an exponential decay rate \emph{independent} of the parameter $\sigma$ (compare Remark \ref{rk:sigma} for the case of the PDE). 

\begin{theorem}[Uniform decay rate for the ODE model]\label{th:ODE}
	Let $(z_0, w_0) \in \R^2$. Let $z \in C^1([0,+\infty))$ the unique solution to \eqref{eq:ODE}. Then there exists constants $\sigma_0, \, R, \, M > 0$, depending on the initial data $(z_0, w_0)$ such that, for every $\sigma \geq \sigma_0$ there exists $z_{\infty} \in [-R,R] \subset \R$ with $\Phi'(z_{\infty}) = 0$, such that 
	\[
	|z(t) - z_{\infty}| \leq M e^{-\frac{1}{2} t} \, .  
	\]
\end{theorem}

\subsection{Case-by-case analysis}

We shall study the possible behavior of the trajectory depending on the initial data $(z_0,w_0)$.

\medskip 

{\itshape Case I: Initial data $z_0 \geq 1$ and $w_0 > 0$}. (Arguing by symmetry, this is analogous to the case $z_0 \leq -1$ and $w_0 < 0$.) Since $w_0 > 0$, by continuity there exists $\delta > 0$ such that $\dot z(t) > 0$ for $t \in (0,\delta)$. This implies that $z(t)$ is strictly increasing for $t \in (0,\delta)$. In particular, $z(t) > z_0 \geq 1$ and $\Phi'(z(t)) = 0$ for $t \in (0,\delta)$. The trajectory $z(t)$ solves the ODE 
\[  \label{eq:ODE Phi'=0}
\begin{cases}
  \ddot z(t) + \dot z(t) = 0 \, , \quad t \in (0,\delta) \, ,\\
  z(0) = z_0 \, , \\
  \dot z(0) = w_0 \, .
\end{cases}  
\]
The unique solution to this problem is given by
\[   \label{eq:formula zk 1}
z(t) = z_0 + w_0 (1-e^{-t}) \, , \quad \text{for every }  t \in [0,\delta) \, . 
\]  
Notice that the function $z_0 + w_0 (1-e^{-t}) \geq 1$ for every $t \geq 0$, since $w_0 > 0$. Thus equality~\eqref{eq:formula zk 1} is extended to $\delta=+\infty$ and $z(t) \to z_0 + w_0$ as $t \to +\infty$. The rate of convergence is exponential with constants independent of $k$ and depending on the size of the initial data. See Figure~\ref{fig:Case I}.
  
\begin{figure}[H]
  \begin{center}
    \includegraphics[scale=0.35]{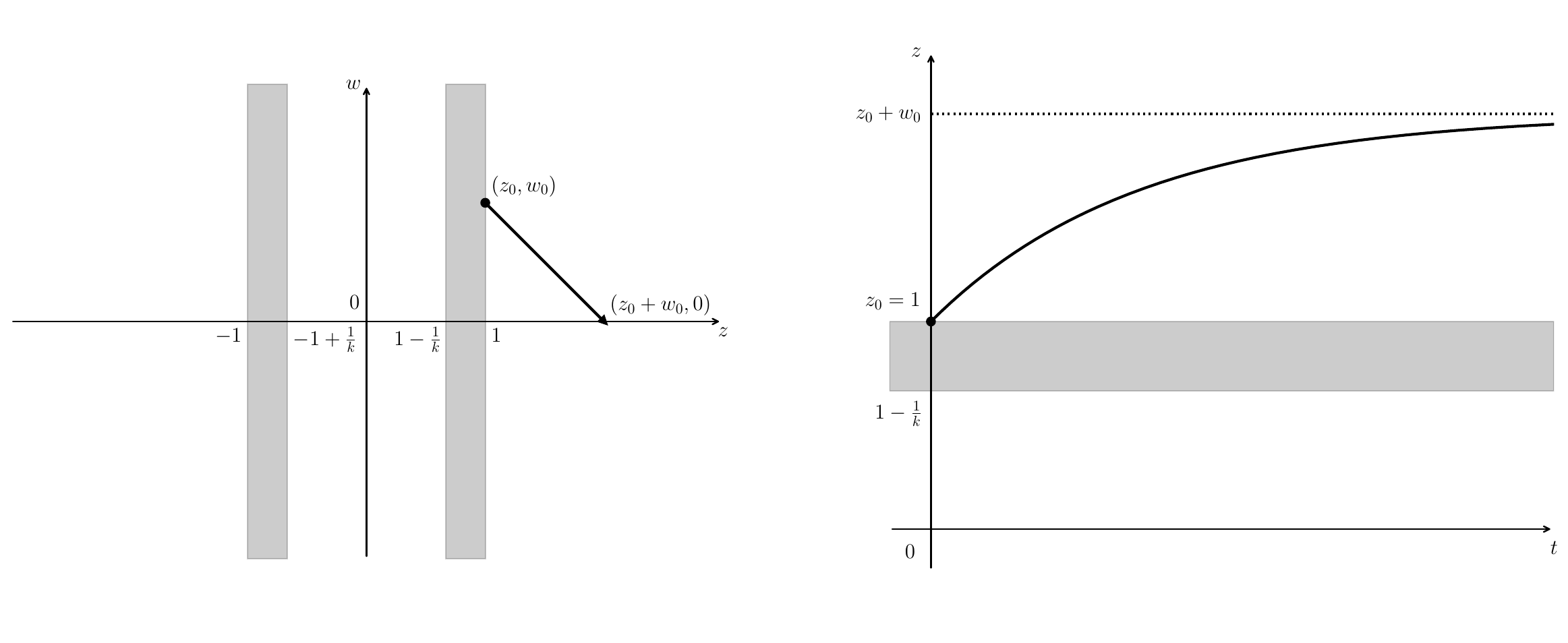}
  \end{center}
  \caption{Example of an evolution described in {\itshape Case I}. On the left, the trajectory is represented in the phase space. On the right, the trajectory is described as a function of time.}
  \label{fig:Case I}
\end{figure}

\vspace{1em}

{\itshape Case II: Initial data $z_0 \geq 1$ and $w_0 = 0$}. (Arguing by symmetry, this is analogous to  the case $z_0 \leq -1$ and $w_0 = 0$.) Since $\Phi'(z_0) = 0$ when $z_0 \geq 1$, the constant trajectory $z(t)  \equiv z_0$ is the unique solution to the ODE~\eqref{eq:ODE}.  See Figure~\ref{fig:Case II}.

\begin{figure}[H]
  \begin{center}
    \includegraphics[scale=0.35]{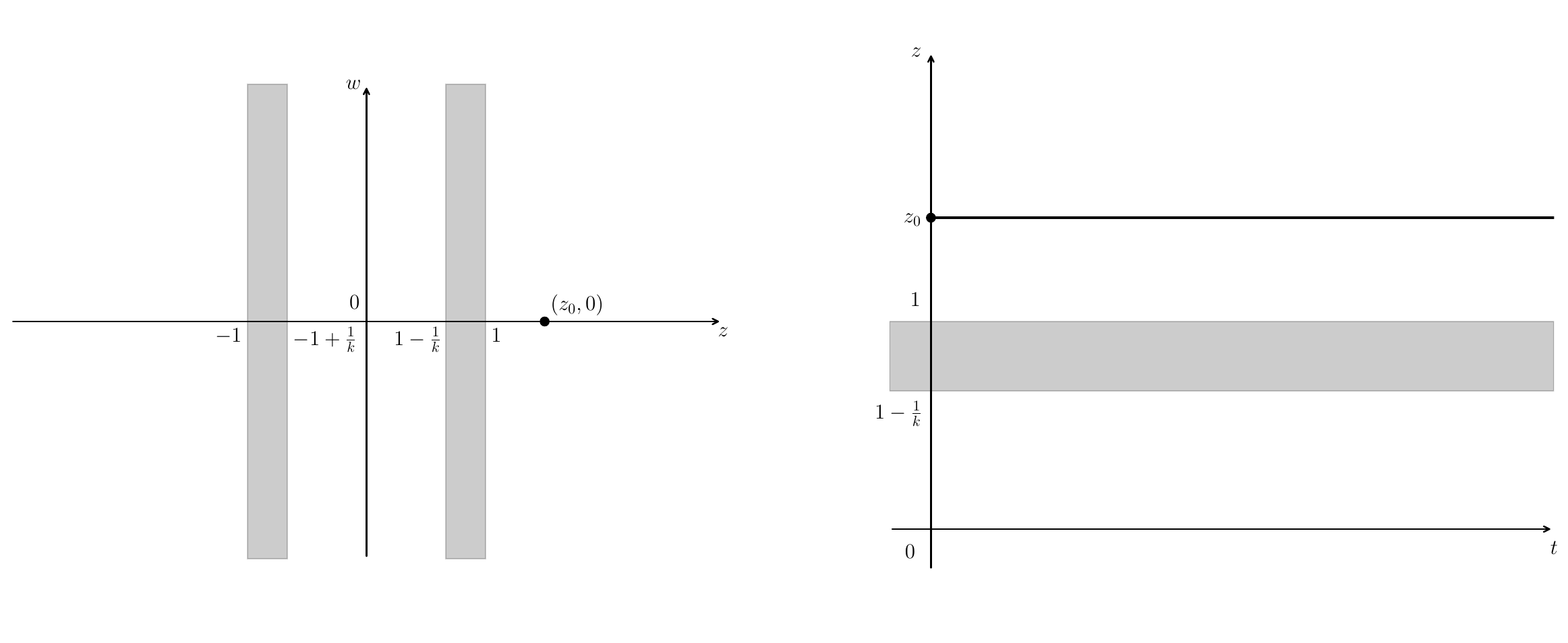}
  \end{center}
  \caption{Example of a constant evolution described in {\itshape Case II}. On the left, the trajectory (a point) is represented in the phase space. On the right, the trajectory is described as a function of time.}
  \label{fig:Case II}
\end{figure}

\vspace{1em}

{\itshape Case III: Initial data $z_0 > 1$ and $w_0 < 0$ with $1-z_0 \leq w_0$}. (Arguing by symmetry, this is analogous to  the case $z_0 < -1$ and $w_0 > 0$ with $w_0 \leq -1-z_0$.) By continuity, there exists $\delta > 0$ such that $z(t) > 1$ for $t \in (0,\delta)$. This implies that $\Phi'(z(t)) = 0$ for $t \in (0,\delta)$. The trajectory $z(t)$ solves the ODE~\eqref{eq:ODE Phi'=0}. The unique solution to this problem is given by
\[   \label{eq:formula zk 2}
z(t) = z_0 + w_0 (1-e^{-t}) = z_0 - |w_0| (1-e^{-t}) \, , \quad \text{for every } t \in [0,\delta) \, .
\] 
We observe that for every $t \geq 0$ 
\[
  z_0 - |w_0| (1-e^{-t}) \geq 1 \iff  e^{-t} \geq \frac{|w_0| + 1 - z_0}{|w_0|} 
\]
and the last inequality holds true since $1-z_0 \leq w_0$. This implies that~\eqref{eq:formula zk 2} can be extended up to $\delta = +\infty$ and $z(t) \to z_0 - |w_0|$ as $t \to +\infty$, with $z_0 - |w_0| \geq 1$. The rate of convergence is exponential with constants independent of $k$. The rate of convergence is exponential with constants independent of $k$ and depending on the size of the initial data. See Figure~\ref{fig:Case III}.

\begin{figure}[H]
  \begin{center}
    \includegraphics[scale=0.35]{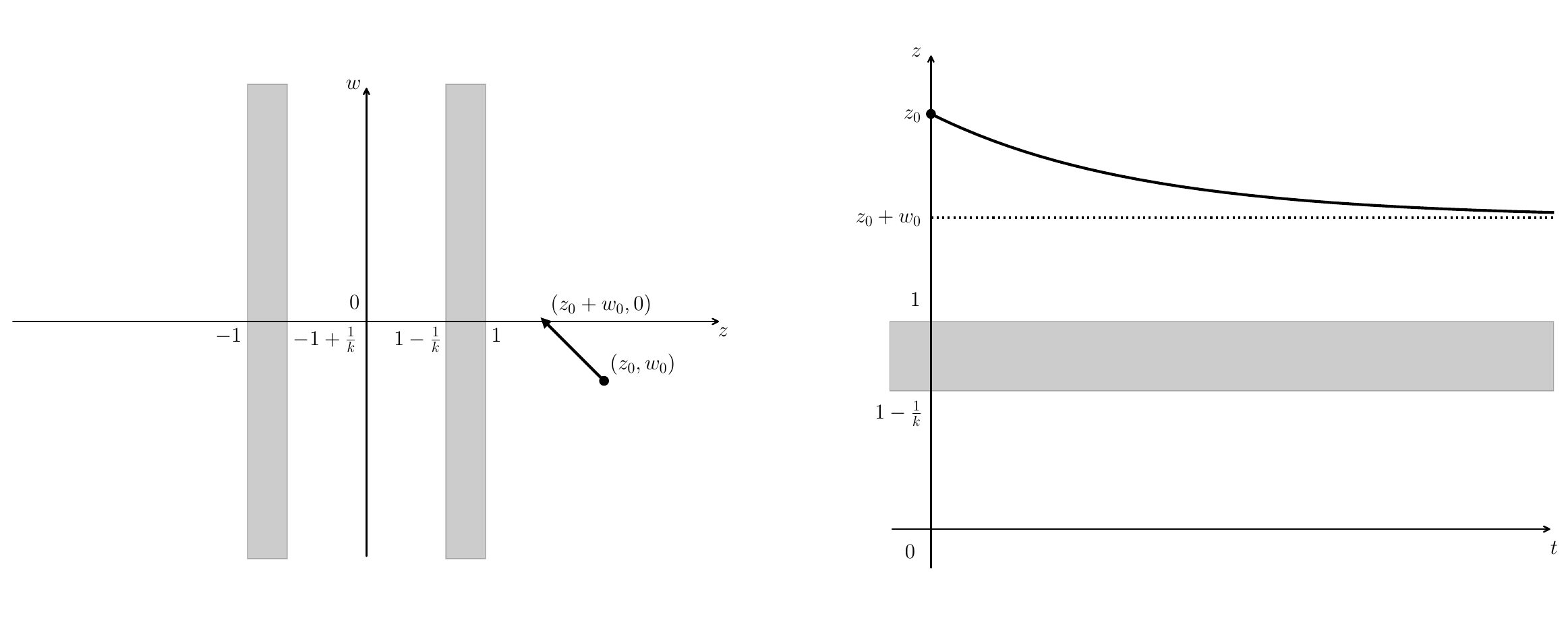}
  \end{center}
  \caption{Example of an evolution described in {\itshape Case III}. On the left, the trajectory is represented in the phase space. On the right, the trajectory is described as a function of time.}
  \label{fig:Case III}
\end{figure}

\vspace{1em}

{\itshape Case IV: Initial data $z_0 > 1$ and $w_0 < 0$ with $ 1-z_0> w_0$}. (Arguing by symmetry, this is analogous to  the case $z_0 < -1$ and $w_0 > 0$ with $w_0 > -1-z_0$.) As in the previous case, there exists $\delta > 0$ such that 
\[   \label{eq:formula zk 3}
z(t) = z_0 + w_0 (1-e^{-t}) = z_0 - |w_0| (1-e^{-t}) \, , \quad \text{for every } t \in [0,\delta) \, .
\] 
Let us find the first time $T_{IV}$ at which the function $z_0 - |w_0| (1-e^{-t})$ reaches the value 1 by solving 
\[
  z_0 - |w_0| (1-e^{-T_{IV}}) = 1 \iff  e^{-T_{IV}} = \frac{|w_0| + 1 - z_0}{|w_0|}  \iff T_{IV} = \log\Big(\frac{|w_0|}{|w_0| + 1 - z_0} \Big)\, .
\]
This implies that~\eqref{eq:formula zk 3} can be extended up to $\delta = T_{IV} = \log\Big(\frac{|w_0|}{|w_0| + 1 - z_0} \Big) > 0$. We stress that $T_{IV}$ depends only on the size of the initial data and is independent of $k$. The exiting speed at time $T_{IV} = \log\Big(\frac{|w_0|}{|w_0| + 1 - z_0} \Big)$ is $\dot z(T_{IV}) = - |w_0| e^{-T_{IV}} < 0$. The future of the evolution after time $T_{IV}$ is obtained by restarting the evolution at time $T_{IV}$ and by considering {\itshape Case~V}. See Figure~\ref{fig:Case IV}.

\begin{figure}[H]
  \begin{center}
    \includegraphics[scale=0.35]{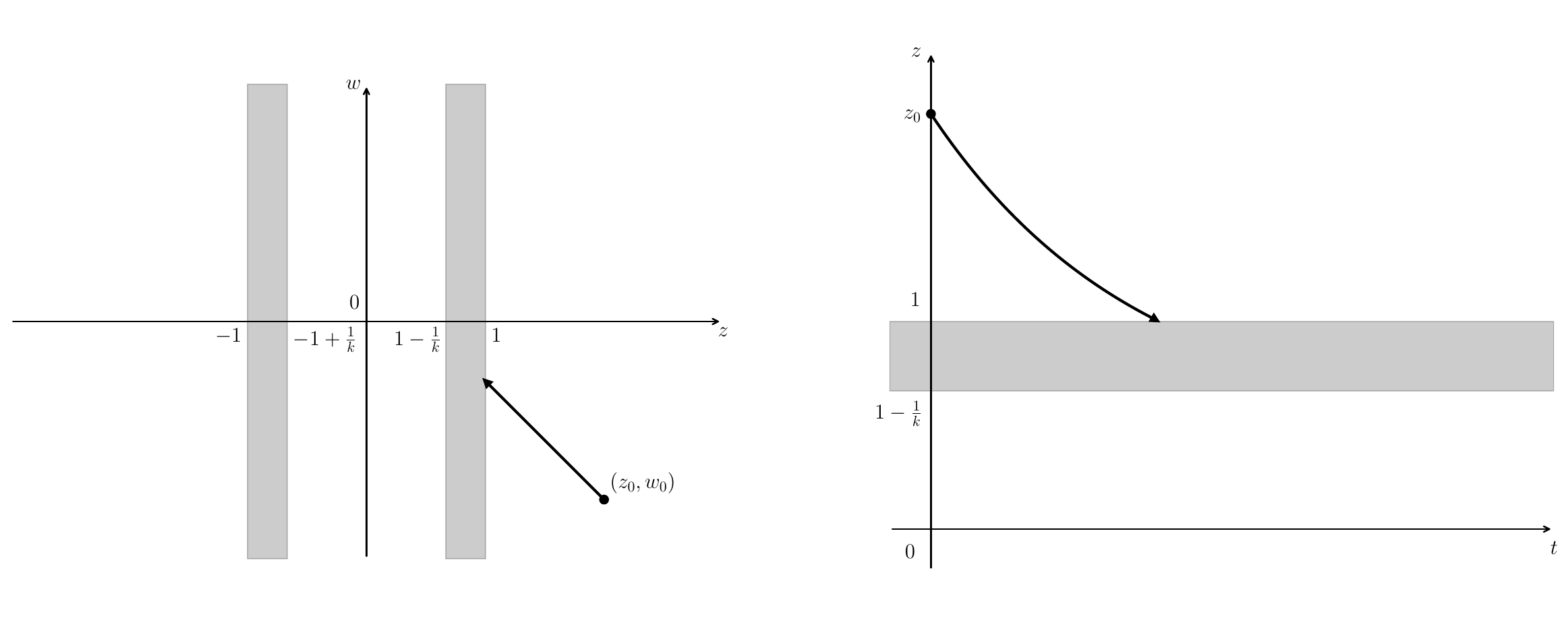}
  \end{center}
  \caption{Example of an evolution described in {\itshape Case IV}. On the left, the trajectory  is represented in the phase space. On the right, the trajectory is described as a function of time.}
  \label{fig:Case IV}
\end{figure}

\vspace{1em}

{\itshape Case V: Initial data $z_0 = 1$ and $w_0 < 0$}. (Arguing by symmetry, this is analogous to  the case $z_0 = -1$ and $w_0 > 0$.) By continuity there exists $\delta > 0$ such that $\dot z(t) < 0$ for $t \in (0,\delta)$. This implies that $z(t)$ is strictly decreasing for $t \in (0,\delta)$. For $\delta > 0$ small enough, we have, in particular, that $1-\frac{1}{\sigma} < z(t) < z_0 = 1$. This implies $\Phi'(z(t)) = 2 (\sigma-1) (1-z(t))$ for $t \in (0,\delta)$ and the trajectory $z(t)$ solves the ODE
\[  \label{eq:ODE Phi'=2 (k-1) (1-z)}
\begin{cases}
  \ddot z(t) + \dot z(t) + 2 (\sigma-1) (1-z(t))= 0 \, , \quad t \in (0,\delta) \, ,\\
  z(0) = z_0 = 1\, , \\
  \dot z(0) = w_0 \, .
\end{cases}  
\]
The unique solution to this problem is given by 
\[ \label{eq:zk exp}
z(t) =  1 + \frac{|w_0|}{\lambda_\sigma + \mu} e^{-\lambda_\sigma t} - \frac{|w_0|}{\lambda_\sigma + \mu} e^{\mu t} \, , \quad \text{for } t \in [0,\delta) \, ,
\] 
where
\[ \label{eq:lambdak muk}
\lambda_\sigma = \frac{1 + \sqrt{1 + 8(k-1)}}{2} > 0  \, , \quad \mu = \frac{- 1 + \sqrt{1 + 8(k-1)}}{2} > 0 \, .
\] 
Let us estimate the first time $T_V^\sigma$ such that the function defined in~\eqref{eq:zk exp} reaches the value $1-\frac{1}{\sigma}$. Let us observe that
\[
  1 + \frac{|w_0|}{\lambda_\sigma + \mu} e^{-\lambda_\sigma t} - \frac{|w_0|}{\lambda_\sigma + \mu} e^{\mu t} < 1 + \frac{|w_0|}{\lambda_\sigma + \mu}   - \frac{|w_0|}{\lambda_\sigma + \mu} e^{\mu t} 
\]
for every $t > 0$. At time $t = 0$ both sides of the inequality are equal to 1. Let us solve 
\[
  1 + \frac{|w_0|}{\lambda_\sigma + \mu}   - \frac{|w_0|}{\lambda_\sigma + \mu} e^{\mu t}  = 1 - \frac{1}{\sigma} \iff  t   = \frac{1}{\mu} \log \Big( 1    +   \frac{\lambda_\sigma + \mu}{k |w_0|} \Big) \, .
\]
This implies that the function $1 + \frac{|w_0|}{\lambda_\sigma + \mu} e^{-\lambda_\sigma t} - \frac{|w_0|}{\lambda_\sigma + \mu} e^{\mu t}$ reaches $1-\frac{1}{\sigma}$ in a time $T_{V}^\sigma$ satisfying
\[
  0 < T_{V}^\sigma < \frac{1}{\mu} \log \Big( 1    +   \frac{\lambda_\sigma + \mu}{k |w_0|} \Big) \, .
\]
We remark that $\frac{1}{\mu}\log ( 1    +   \frac{\lambda_\sigma + \mu}{k |w_0|} ) \to 0$ as $k \to +\infty$. In particular, $T_{V}^\sigma < 1$ for $k$ large enough, depending only on the size of the initial data. The exiting speed at time $T_{V}^\sigma$ is $\dot z(T_{V}^\sigma) = -\frac{\lambda_\sigma |w_0|}{\lambda_\sigma + \mu} e^{-\lambda_\sigma T_{V}^\sigma} - \frac{\mu |w_0|}{\lambda_\sigma + \mu} e^{\mu T_{V}^\sigma} < 0$. See Figure~\ref{fig:Case V}. The future of the evolution after time $T_V^\sigma$ is obtained by restarting the evolution at time $T_V^\sigma$ and by considering {\itshape Case VI}.

\begin{figure}[H]
  \begin{center}
    \includegraphics[scale=0.35]{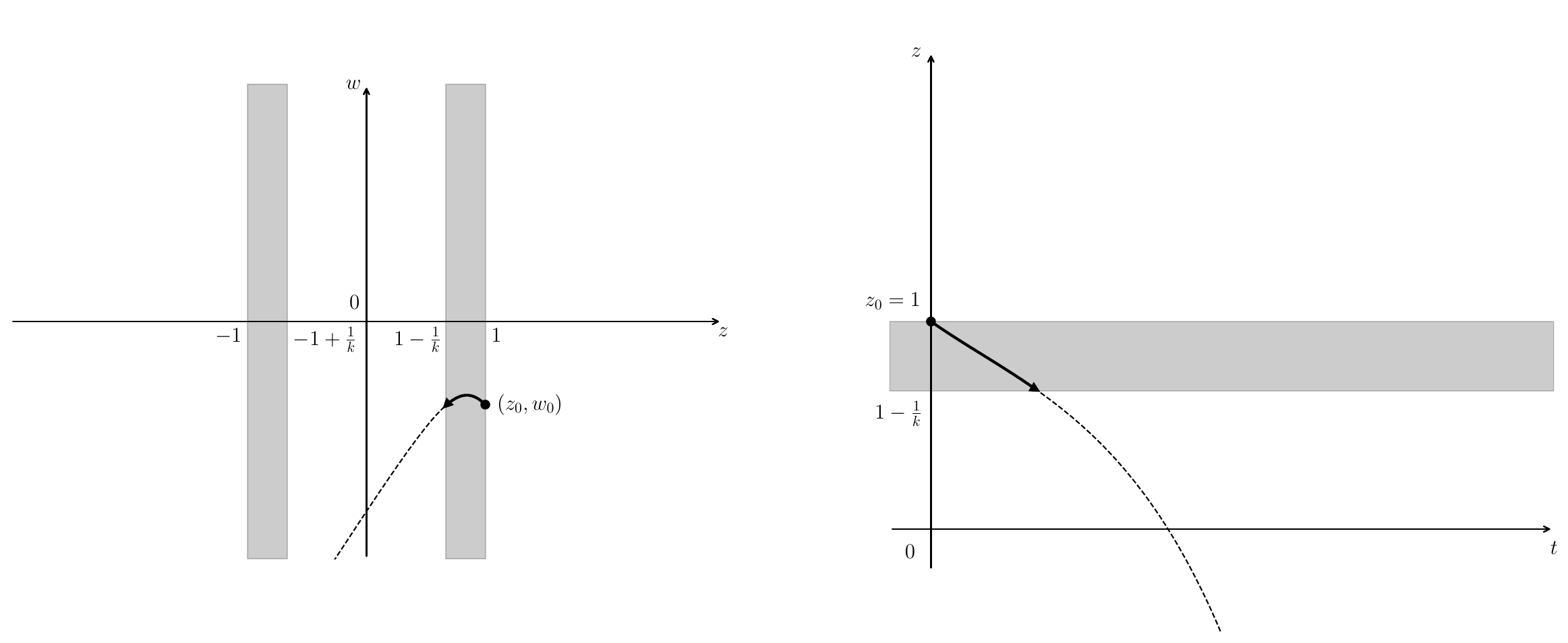}
  \end{center}
  \caption{Example of an evolution described in {\itshape Case V}. On the left, the trajectory  is represented in the phase space. On the right, the trajectory is described as a function of time. The solid line depicts $z(t)$ while the dashed line represents the function $1 + \frac{|w_0|}{\lambda_\sigma + \mu} e^{-\lambda_\sigma t} - \frac{|w_0|}{\lambda_\sigma + \mu} e^{\mu t}$ beyond time $T_{V}^\sigma$.}
  \label{fig:Case V}
\end{figure}

\vspace{1em}

We stress that the cases including the initial data satisfying $1 - \frac{1}{\sigma} \leq z_0 < 1$ or $-1 < z_0 \leq - 1 + \frac{1}{\sigma}$ are unnecessary, since for $k$ large enough (depending on the initial data), these are not satisfied. Thus, we consider directly the case $-1+\frac{1}{\sigma} < z_0 < 1-\frac{1}{\sigma}$. 

\vspace{1em}

{\itshape Case VI: Initial data $-1+\frac{1}{\sigma} \leq z_0 \leq 1-\frac{1}{\sigma}$ and $w_0 \in \R$}. Several possibilities might occur: 
\begin{itemize}
  \item If $-1+\frac{1}{\sigma} < z_0 < 1-\frac{1}{\sigma}$, by continuity there exists $\delta > 0$ such that $-1+\frac{1}{\sigma} < z(t) < 1-\frac{1}{\sigma}$.  
  \item If $z_0 = 1 - \frac{1}{\sigma}$ and $w_0 > 0$, we fall into {\itshape Case VII} and we do not study this case here as the trajectory immediately exits from the interval $[-1+\frac{1}{\sigma}, 1-\frac{1}{\sigma}]$.
  \item If $z_0 = 1 - \frac{1}{\sigma}$ and $w_0 = 0$, then there exists $\delta > 0$ such that $\dot z(t) \neq 0$ for $t \in (0,\delta)$ (for, otherwise, $1 - \frac{1}{\sigma}$ would be an equilibrium point). It cannot occur that $\dot z(t) > 0$. Indeed, this would imply that $1-\frac{1}{\sigma} < z(t) < 1$ for $t \in (0,\delta)$ with $\delta$ small enough. This would give that $\Phi'(z(t)) = 2 (k-1) (1-z(t))$ for $t \in (0,\delta)$ and the trajectory $z(t)$ would solve the ODE
  \[   
  \begin{cases}
    \ddot z(t) + \dot z(t) + 2 (k-1) (1-z(t))= 0 \, , \quad t \in (0,\delta) \, ,\\
    z(0) = z_0 = 1 - \frac{1}{\sigma}\, , \\
    \dot z(0) = 0 \, .
  \end{cases}  
  \]
  The explicit solution to this problem is
  \[  
  z(t) =  1 - \frac{\frac{\mu }{\sigma}}{\lambda_\sigma + \mu }  e^{-\lambda_\sigma t} - \frac{\frac{\lambda_\sigma}{\sigma}}{ \lambda_\sigma + \mu} e^{\mu t} \, , \quad \text{for } t \in [0,\delta) \, ,
  \] 
  where $\lambda_\sigma$ and $\mu$ are defined as in~\eqref{eq:lambdak muk}. However, the derivative of this function is 
  \[
  \dot z(t) =  \frac{\frac{\lambda_\sigma \mu }{\sigma}}{\lambda_\sigma + \mu } \Big( e^{-\lambda_\sigma t} - e^{\mu t} \Big) < 0 \, , \quad \text{for } t \in (0,\delta) \, ,
  \]
  which would lead to a contradiction. We conclude that $\dot z(t) < 0$ for $t \in (0,\delta)$, hence $z(t)$ is strictly decreasing and, in particular, $-1+\frac{1}{\sigma} < z(t) < z_0 = 1-\frac{1}{\sigma}$ for $t \in (0,\delta)$  with $\delta$ small enough.
  \item If $z_0 = 1 - \frac{1}{\sigma}$ and $w_0 < 0$, then by continuity there exists $\delta > 0$ such that $\dot z(t) < 0$ for $t \in (0,\delta)$. Hence $z(t)$ is strictly decreasing and, in particular, $-1+\frac{1}{\sigma} < z(t) < z_0 = 1-\frac{1}{\sigma}$ for $t \in (0,\delta)$  with $\delta$ small enough.
  \item If $z_0 = - 1 + \frac{1}{\sigma}$ and $w_0 < 0$, we fall into {\itshape Case VII} and we do not study this case here.
  \item If $z_0 = - 1 + \frac{1}{\sigma}$ and $w_0 = 0$, this is analogous to the case $z_0 = 1 - \frac{1}{\sigma}$ and $w_0 = 0$.
  \item If $z_0 = - 1 + \frac{1}{\sigma}$ and $w_0 > 0$, this is analogous to the case $z_0 = 1 - \frac{1}{\sigma}$ and $w_0 < 0$.
\end{itemize}
In all cases, $-1+\frac{1}{\sigma} < z(t) < z_0 = 1-\frac{1}{\sigma}$ for $t \in (0,\delta)$, thus $\Phi'(z(t)) = 2 z(t)$ and $z(t)$ solves the ODE 
\[  \label{eq:ODE Phi'=2 z}
\begin{cases}
  \ddot z(t) + \dot z(t) + 2 z(t)= 0 \, , \quad t \in (0,\delta) \, ,\\
  z(0) = z_0 \, , \\
  \dot z(0) = w_0 \, .
\end{cases}  
\]
The unique solution to this problem is given by 
\[ \label{eq:formula zk 4}
z(t) = z_0 e^{-\frac{1}{2} t} \cos( \omega t)  + \frac{2 w_0 +  z_0   }{2\omega} e^{-\frac{1}{2} t} \sin( \omega t)   \, , \quad \text{for every } t \in [0,\delta) \, .
\]
where $\omega = \frac{\sqrt{7}}{2}$. We distinguish several possibilities, depending on the extreme values attained by the function $z_0 e^{-\frac{1}{2} t} \cos( \omega t)  + \frac{2 w_0 +  z_0   }{2\omega} e^{-\frac{1}{2} t} \sin( \omega t)$ for $t \in [0,+\infty)$. The explicit values of the extreme values attained by this function are not relevant for the discussion. However, it is possible to compute them by considering the first time where the derivative vanishes. The first critical point is $\bar t =  \frac{1}{\omega}\arctan\Big( \frac{4\omega  w_0 }{ (4  \omega^2 + 1) z_0 + 2 w_0    } \Big)$ or $\bar t =  \frac{1}{\omega}\arctan\Big( \frac{4\omega  w_0 }{ (4  \omega^2 + 1) z_0 + 2 w_0    } \Big) + \frac{\pi}{\omega}$ (depending on the sign of $w_0$). All  critical points are of the form $t = \bar t + \frac{n}{\omega} \pi$, $n \in \N$. The values attained at these critical points are 
\[
  e^{-\frac{1}{2}  \frac{n}{\omega}} \Big( z_0 e^{-\frac{1}{2} \bar t} \cos( \omega \bar t)  + \frac{2 w_0 +  z_0   }{2\omega} e^{-\frac{1}{2} \bar t} \sin( \omega \bar t)\Big) \, .
\]
Note that their absolute value is therefore damped as $n$ increases. For this reason, the extreme values are attained either at $t = 0$ and $t = \bar t$ or at $t = \bar t$ and $t = \bar t + \frac{\pi}{\omega}$.
\begin{itemize}
  \item The extreme value at time $t = \bar t$ lies in the interval $( -1+\frac{1}{\sigma}, 1-\frac{1}{\sigma})$. See Figure~\ref{fig:Case VI-a}. Since at time $t = 0$ we already have that $-1+\frac{1}{\sigma} < z_0 < 1-\frac{1}{\sigma}$ and all future extreme values are damped, this implies that the function $z_0 e^{-\frac{1}{2} t} \cos( \omega t)  + \frac{2 w_0 +  z_0   }{2\omega} e^{-\frac{1}{2} t} \sin( \omega t)$ lies in the interval $(-1+\frac{1}{\sigma}, 1-\frac{1}{\sigma})$ for every $t \in [0,+\infty)$. In this case, formula~\eqref{eq:formula zk 4} can be extended up to $\delta = +\infty$. The solution satisfies $z(t) \to 0$ as $t \to +\infty$. The rate of convergence is exponential with constants independent of $k$ and depending on the size of the initial data. 
   
  \begin{figure}[H]
    \begin{center}
      \includegraphics[scale=0.35]{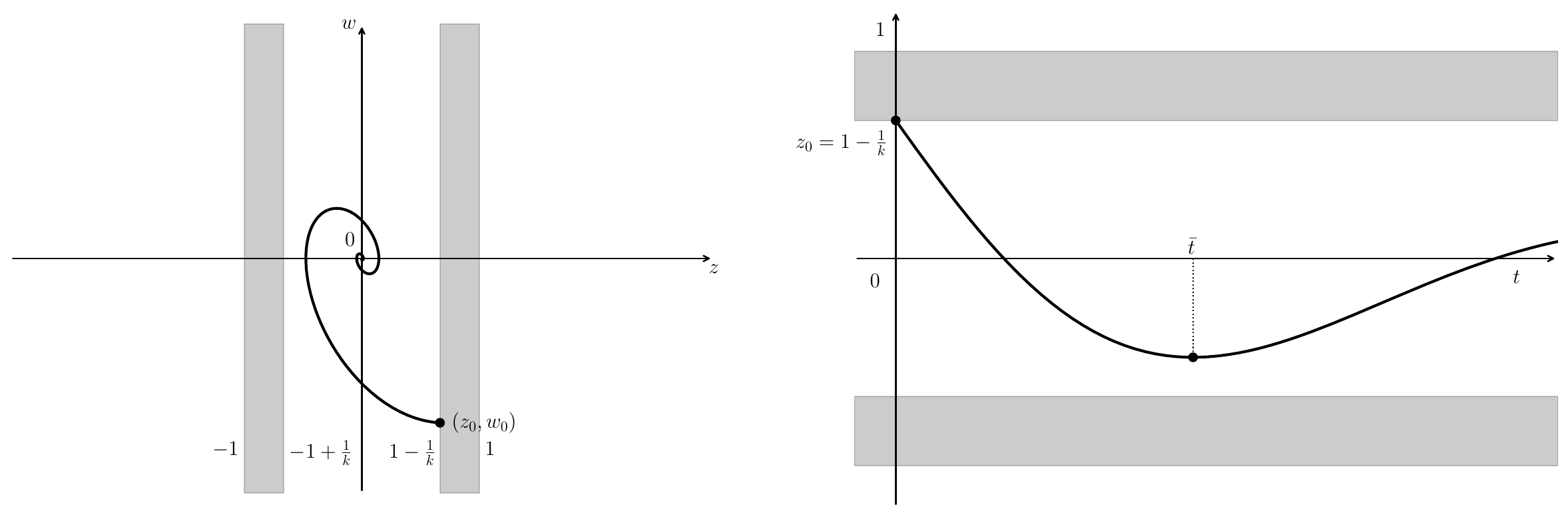}
    \end{center}
    \caption{Example of an evolution described in {\itshape Case VI} when the extreme value at time $t = \bar t$ lies in the interval $( -1+\frac{1}{\sigma}, 1-\frac{1}{\sigma})$. On the left, the trajectory  is represented in the phase space. On the right, the trajectory is described as a function of time.}
    \label{fig:Case VI-a}
  \end{figure}

  \item The extreme value at time $t = \bar t$ lies outside the interval $( -1+\frac{1}{\sigma}, 1-\frac{1}{\sigma})$. See Figure~\ref{fig:Case VI-b}. This means that either the global minimum is less than or equal to $-1+\frac{1}{\sigma}$ or the global maximum is greater that or equal to $1-\frac{1}{\sigma}$. In either case, we extend~\eqref{eq:formula zk 4} up to $\delta = T_{VI}^\sigma$, where $T_{VI}^\sigma$ is the first exiting time such that the function $|z_0 e^{-\frac{1}{2} t} \cos( \omega t)  + \frac{2 w_0 +  z_0   }{2\omega} e^{-\frac{1}{2} t} \sin( \omega t)|$ reaches the value $1-\frac{1}{\sigma}$. Note that $T_{VI}^\sigma$ is bounded by $\bar t$. At time $T_{VI}^\sigma$, the sign of the exiting speed $\dot z(T_{VI}^\sigma)$ depends on whether $z(T_{VI}^\sigma) = 1 - \frac{1}{\sigma}$ (positive, in this case) or $z(T_{VI}^\sigma) = -1 + \frac{1}{\sigma}$  (negative, in this case). After time $T_{VI}^\sigma$, the future of the evolution is obtained by restarting the evolution at time $T_{VI}^\sigma$ and by considering {\itshape Case VII}.
  
  \begin{figure}[H]
    \begin{center}
      \includegraphics[scale=0.35]{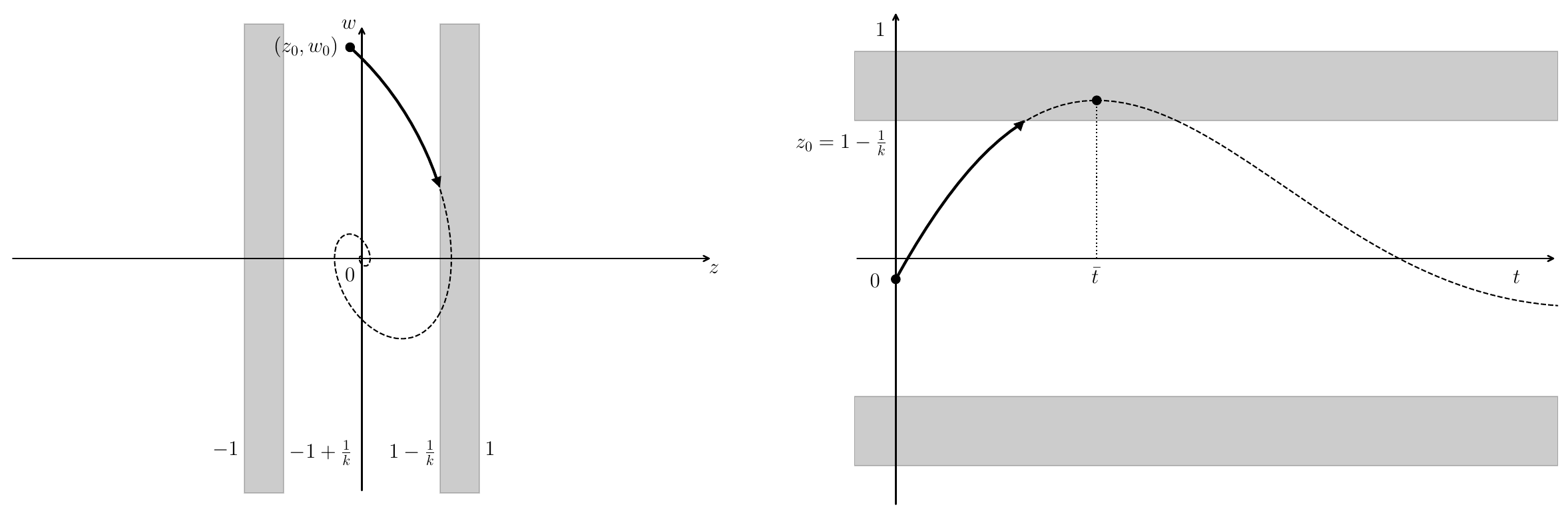}
    \end{center}
    \caption{Example of an evolution described in {\itshape Case VI} when the extreme value at time $t = \bar t$ lies outside the interval $( -1+\frac{1}{\sigma}, 1-\frac{1}{\sigma})$. On the left, the trajectory  is represented in the phase space. On the right, the trajectory is described as a function of time. The solid line depicts $z(t)$ while the dashed line represents the function $z_0 e^{-\frac{1}{2} t} \cos( \omega t)  + \frac{2 w_0 +  z_0   }{2\omega} e^{-\frac{1}{2} t} \sin( \omega t)$  beyond time $T_{VI}^\sigma$.}
    \label{fig:Case VI-b}
  \end{figure}
\end{itemize}

\vspace{1em}
  
{\itshape Case VII: Initial data $z_0 = 1 - \frac{1}{\sigma}$ and $w_0 > 0$.} (Arguing by symmetry, this is analogous to  the case $z_0 = - 1 + \frac{1}{\sigma}$ and $w_0 < 0$.) Since $w_0 > 0$, by continuity there exists $\delta > 0$ such that $\dot z(t) > 0$ for $t \in (0,\delta)$. Hence, $z(t)$ is strictly increasing and, in particular, $1 - \frac{1}{\sigma} = z_0 < z(t) < 1$ for $t \in (0,\delta)$ if $\delta$ is small enough. This implies $\Phi'(z(t)) = 2 (\sigma-1) (1-z(t))$ for $t \in (0,\delta)$ and the trajectory $z(t)$ solves the ODE
\[  \label{eq:ODE Phi'=2 (11) (1-z) 2}
\begin{cases}
  \ddot z(t) + \dot z(t) + 2 (\sigma-1) (1-z(t))= 0 \, , \quad t \in (0,\delta) \, ,\\
  z(0) = z_0 = 1 - \frac{1}{\sigma}\, , \\
  \dot z(0) = w_0 \, .
\end{cases}  
\]
The explicit solution to this problem is
\[ \label{eq:zk exp 2}
z(t) =  1 - \frac{w_0  + \frac{\mu }{\sigma}}{\lambda_\sigma + \mu }  e^{-\lambda_\sigma t} + \frac{ w_0  - \frac{\lambda_\sigma}{\sigma}}{ \lambda_\sigma + \mu} e^{\mu t} \, , \quad \text{for } t \in [0,\delta) \, ,
\] 
where $\lambda_\sigma$ and $\mu$ are defined as in~\eqref{eq:lambdak muk}. Let us estimate the time $T_{VII}^\sigma$ needed by the function $1 - \frac{w_0  + \frac{\mu }{\sigma}}{\lambda_\sigma + \mu }  e^{-\lambda_\sigma t} + \frac{ w_0  - \frac{\lambda_\sigma}{\sigma}}{ \lambda_\sigma + \mu} e^{\mu t} $ to reach the value 1 by solving 
\[
  1 - \frac{w_0  + \frac{\mu }{\sigma}}{\lambda_\sigma + \mu }  e^{-\lambda_\sigma T_{VII}^\sigma} + \frac{ w_0  - \frac{\lambda_\sigma}{\sigma}}{ \lambda_\sigma + \mu} e^{\mu T_{VII}^\sigma}  = 1 \iff    T_{VII}^\sigma = \frac{1}{\lambda_\sigma + \mu}\log \Big( \frac{w_0  + \frac{\mu }{\sigma}}{ w_0  - \frac{\lambda_\sigma}{\sigma}}     \Big) \, .
\] 
Note that $w_0 - \frac{\lambda_\sigma}{\sigma} > 0$ for $k$ large enough depending on the size of the initial data. We remark that $T_{VII}^\sigma \to 0$ as $k \to +\infty$, hence $T_{VII}^\sigma < 1$ for $k$ large enough, depending only on the size of the initial data. The exiting speed is $\dot z(T_{VII}^\sigma) = \lambda_\sigma \frac{w_0  + \frac{\mu }{\sigma}}{\lambda_\sigma + \mu }  e^{-\lambda_\sigma T_{VII}^\sigma} + \mu \frac{ w_0  - \frac{\lambda_\sigma}{\sigma}}{ \lambda_\sigma + \mu} e^{\mu T_{VII}^\sigma}> 0$. Thus the future evolution is obtained by restarting the evolution at time $T_{VII}^\sigma$ and by considering {\itshape Case I}. 

\begin{figure}[H]
  \begin{center}
    \includegraphics[scale=0.35]{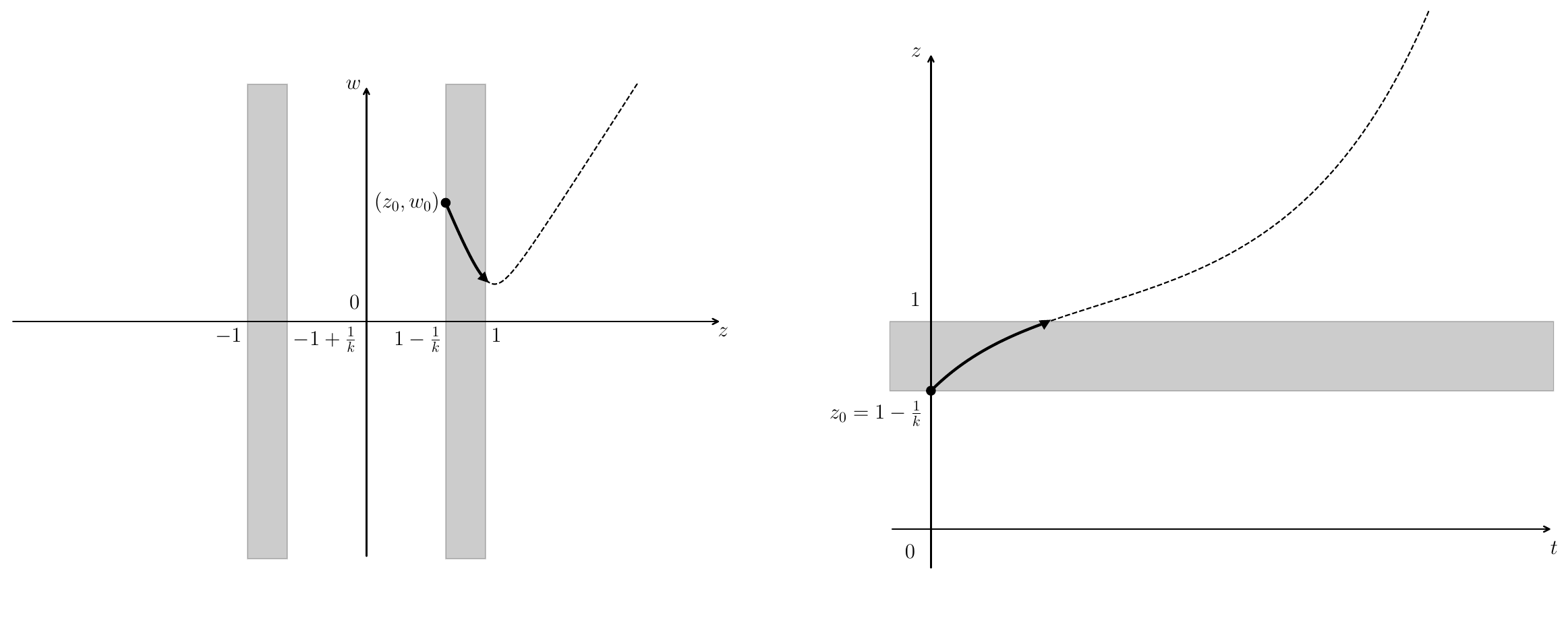}
  \end{center}
  \caption{Example of an evolution described in {\itshape Case VII}. On the left, the trajectory  is represented in the phase space. On the right, the trajectory is described as a function of time. The solid line depicts $z(t)$ while the dashed line represents the function $1 - \frac{w_0  + \frac{\mu }{\sigma}}{\lambda_\sigma + \mu }  e^{-\lambda_\sigma t} + \frac{ w_0  - \frac{\lambda_\sigma}{\sigma}}{ \lambda_\sigma + \mu} e^{\mu t}$ beyond time $T_{VII}^\sigma$.}
  \label{fig:Case VII}
\end{figure}

\subsection{Proof of the decay rate}

Combining all the cases discussed above, we can prove Theorem \ref{th:ODE}

\begin{proof}[Proof of Theorem \ref{th:ODE}]
By the previous discussion, we deduce that $z(t)$ transits in the intervals $(1-\frac{1}{\sigma},1)$ and $(-1,-1+\frac{1}{\sigma})$ at most twice. The time spent in these intervals is either $T_{V}^\sigma$ or $T_{VII}^\sigma$. In either case, the time spent in these intervals is bounded by a constant depending only on the size of the initial data. The time spent to do a transition from $(1-\frac{1}{\sigma},1)$ to $(-1,-1+\frac{1}{\sigma})$  (or viceversa) is $T_{VI}^\sigma$, which is bounded by the size of the initial data. This transition may occur only once. Then the solution $z(t)$ lies in only one of the intervals $[-1+\frac{1}{\sigma}, 1-\frac{1}{\sigma}]$ $[1,+\infty)$, $(-\infty,1]$ for the rest of the evolution, converging to equilibrium exponentially fast. The worst exponential rate is $\frac{1}{2}$. The constant $M$ depends on the bounds of the times $T_{V}^\sigma$, $T_{VI}^\sigma$, and $T_{VII}^\sigma$.
\end{proof}

\vspace{5mm}
\section*{Acknowledgments}

G. M. Coclite, N. De Nitti, F. Maddalena, and G. Orlando are members of the Gruppo Nazionale per l'Analisi Matematica, la Probabilità e le loro Applicazioni (GNAMPA) of the Istituto Nazionale di Alta Matematica (INdAM).

G. M. Coclite, F. Maddalena, and G. Orlando were supported by the Italian Ministry of University and Research under the Programme “Department of Excellence” Legge 232/2016 (Grant No. CUP - D93C23000100001).

G. M. Coclite, F. Maddalena, and G. Orlando have received funding from the INdAM - GNAMPA 2023 Project Codice CUP E53C22001930001, ``Modelli variazionali ed evolutivi per problemi di adesione e di contatto''.

G. M. Coclite, F. Maddalena, and G. Orlando have been partially supported by the Research Project of National Relevance ``Evolution problems involving interacting scales'' granted by the Italian Ministry of Education, University and Research (MIUR Prin 2022, project code 2022M9BKBC, Grant No. CUP D53D23005880006).

G. M. Coclite has been partially supported by the Project  funded  under  the  National  Recovery  and  Resilience  Plan  (NRRP),  Mission  4 Component  2  Investment  1.4 -Call  for  tender  No.  3138  of  16/12/2021
of  Italian  Ministry  of  University and Research funded by the European Union -NextGenerationEUoAward  Number:  
CN000023,  Concession  Decree  No.  1033  of  17/06/2022  adopted  by  the  Italian Ministry of University and Research, CUP: D93C22000410001, Centro Nazionale per la Mobilit\`a Sostenibile
and the Italian
Ministry of Education, University and Research under the Programme Department of Excellence Legge 232/2016 (Grant No. CUP - D93C23000100001).

G. Orlando has been supported by the project ``Approccio integrato e predittivo per il controllo della criminalità marittima'' in the program ``Research for Innovation'' (REFIN) - POR Puglia FESR FSE 2014-2020, Codice CUP: D94I20001410008.

E. Zuazua has been funded by the Alexander von Humboldt-Professorship program, the ModConFlex Marie Curie Action, HORIZON-MSCA-2021-DN-01, the COST Action MAT-DYN-NET, the Transregio 154 Project ``Mathematical Modelling, Simulation and Optimization Using the Example of Gas Networks'' of the DFG, grants PID2020-112617GB-C22 and TED2021-131390B-I00 of MINECO (Spain), and by the Madrid Goverment -- UAM Agreement for the Excellence of the University Research Staff in the context of the V PRICIT (Regional Programme of Research and Technological Innovation).

\vspace{5mm}
\bibliographystyle{abbrv}  
\bibliography{bibliography.bib} 

\begin{thebibliography}{10}

\bibitem{Bal78}
J.~M. Ball.
\newblock On the asymptotic behavior of generalized processes, with
  applications to nonlinear evolution equations.
\newblock {\em J. Differ. Equations}, 27:224--265, 1978.

\bibitem{Bal04}
J.~M. Ball.
\newblock Global attractors for damped semilinear wave equations.
\newblock {\em Discrete Contin. Dyn. Syst.}, 10(1-2):31--52, 2004.

\bibitem{Brez}
H.~Brezis.
\newblock {\em Functional analysis, {S}obolev spaces and partial differential
  equations}.
\newblock Universitext. Springer, New York, 2011.

\bibitem{CazHar}
T.~Cazenave and A.~Haraux.
\newblock {\em An introduction to semilinear evolution equations}, volume~13 of
  {\em Oxford Lecture Series in Mathematics and its Applications}.
\newblock The Clarendon Press, Oxford University Press, New York, 1998.

\bibitem{CocDevFloLigMad23}
G.~M. Coclite, G.~Devillanova, G.~Florio, M.~Ligab\`o, and F.~Maddalena.
\newblock Thermo-elastic waves in a model with nonlinear adhesion.
\newblock {\em Nonlinear Anal.}, 232:Paper No. 113265, 16, 2023.

\bibitem{CocFloLigMad17}
G.~M. Coclite, G.~Florio, M.~Ligab\`o, and F.~Maddalena.
\newblock Nonlinear waves in adhesive strings.
\newblock {\em SIAM J. Appl. Math.}, 77(2):347--360, 2017.

\bibitem{CocFloLigMad19}
G.~M. Coclite, G.~Florio, M.~Ligab\`o, and F.~Maddalena.
\newblock Adhesion and debonding in a model of elastic string.
\newblock {\em Comput. Math. Appl.}, 78(6):1897--1909, 2019.

\bibitem{Daf78}
C.~M. Dafermos.
\newblock Asymptotic behavior of solutions of evolution equations.
\newblock In {\em Nonlinear evolution equations ({P}roc. {S}ympos., {U}niv.
  {W}isconsin, {M}adison, {W}is., 1977)}, volume~40 of {\em Publication of the
  Mathematics Research Center, University of Wisconsin}, pages 103--123.
  Academic Press, New York-London, 1978.

\bibitem{Dic76}
R.~W. Dickey.
\newblock Stability theory for the damped sine-{Gordon} equation.
\newblock {\em SIAM J. Appl. Math.}, 30:248--262, 1976.

\bibitem{GhiGobHar16-1}
M.~Ghisi, M.~Gobbino, and A.~Haraux.
\newblock Finding the exact decay rate of all solutions to some second order
  evolution equations with dissipation.
\newblock {\em J. Funct. Anal.}, 271(9):2359--2395, 2016.

\bibitem{GhiGobHar16-2}
M.~Ghisi, M.~Gobbino, and A.~Haraux.
\newblock Optimal decay estimates for the general solution to a class of
  semi-linear dissipative hyperbolic equations.
\newblock {\em J. Eur. Math. Soc. (JEMS)}, 18(9):1961--1982, 2016.

\bibitem{Hal69}
J.~K. Hale.
\newblock Dynamical systems and stability.
\newblock {\em Journal of Mathematical Analysis and Applications},
  26(1):39--59, 1969.

\bibitem{Har86}
A.~Haraux.
\newblock Asymptotics for some nonlinear hyperbolic equations with a
  one-dimensional set of rest points.
\newblock {\em Bol. Soc. Bras. Mat.}, 17(1):51--65, 1986.

\bibitem{Har}
A.~Haraux.
\newblock Semi-linear hyperbolic problems in bounded domains.
\newblock {\em Math. Rep.}, 3(1):i--xxiv and 1--281, 1987.

\bibitem{Har05}
A.~Haraux.
\newblock Slow and fast decay of solutions to some second order evolution
  equations.
\newblock {\em J. Anal. Math.}, 95:297--321, 2005.

\bibitem{HarZua88}
A.~Haraux and E.~Zuazua.
\newblock Decay estimates for some semilinear damped hyperbolic problems.
\newblock {\em Arch. Ration. Mech. Anal.}, 100(2):191--206, 1988.

\bibitem{Las67}
J.~P. LaSalle.
\newblock An invariance principle in the theory of stability.
\newblock Differ. {Equations} dynam. {Systems}, {Proc}. {Int}. {Sympos}.
  {Puerto} {Rico} 1965, 277-286 (1967)., 1967.

\bibitem{MadPerPugTru09}
F.~Maddalena, D.~Percivale, G.~Puglisi, and L.~Truskinovsky.
\newblock Mechanics of reversible unzipping.
\newblock {\em Contin. Mech. Thermodyn.}, 21(4):251--268, 2009.

\bibitem{Mar00}
P.~Martinez.
\newblock Stabilization for the wave equation with {N}eumann boundary condition
  by a locally distributed damping.
\newblock In {\em Contr\^{o}le des syst\`emes gouvern\'{e}s par des
  \'{e}quations aux d\'{e}riv\'{e}es partielles ({N}ancy, 1999)}, volume~8 of
  {\em ESAIM Proc.}, pages 119--136. Soc. Math. Appl. Indust., Paris, 2000.

\bibitem{Web79}
G.~F. Webb.
\newblock A bifurcation problem for a nonlinear hyperbolic partial differential
  equation.
\newblock {\em SIAM J. Math. Anal.}, 10:922--932, 1979.

\bibitem{Zua88}
E.~Zuazua.
\newblock Stability and decay for a class of nonlinear hyperbolic problems.
\newblock {\em Asymptotic Anal.}, 1(2):161--185, 1988.

\bibitem{Zua90}
E.~Zuazua.
\newblock Exponential decay for the semilinear wave equation with locally
  distributed damping.
\newblock {\em Comm. Partial Differential Equations}, 15(2):205--235, 1990.

\bibitem{Zua91}
E.~Zuazua.
\newblock Exponential decay for the semilinear wave equation with localized
  damping in unbounded domains.
\newblock {\em J. Math. Pures Appl. (9)}, 70(4):513--529, 1991.

\end{thebibliography}
\vfill 
\end{document}